\documentclass[a4paper, reqno]{amsart}
\usepackage{hyperref,url}
\usepackage{a4wide}
\usepackage{amsthm, amssymb, amsmath, latexsym, upgreek,accents}
\usepackage{amsfonts, mathrsfs, color}
\usepackage{graphicx,soul}

\usepackage{tikz}
\numberwithin{equation}{section}

\theoremstyle{plain}                     
\begingroup
\newtheorem{theorem}{Theorem}[section]       
\newtheorem{prop}[theorem]{Proposition}     
\newtheorem{cor}[theorem]{Corollary}        
\newtheorem{lemma}[theorem]{Lemma}      
\endgroup
      
\theoremstyle{definition}                
\begingroup

\newtheorem{remark}[theorem]{Remark}
\endgroup

\newcommand{\R}{\mathbb R}
\newcommand{\N}{\mathbb N}

\newcommand{\C}{\mathbb C}
\newcommand{\e}{\varepsilon}
\newcommand{\mtwo}{\mathbb M^{2\times2}}
\newcommand{\mpos}{\mathbb M^{2\times2}_{\sym,+}}
\newcommand{\HH}{{\mathcal H}}
\newcommand{\LL}{{\mathcal L}}
\newcommand{\Om}{\widehat\Psi}

\def\P{\mathcal P}
\def\S{\mathcal S}

\DeclareMathOperator\supp{supp}

\DeclareMathOperator\sym{sym}
\DeclareMathOperator\Div{div}
\DeclareMathOperator\curl{curl}
\DeclareMathOperator\cpc{cap}
\DeclareMathOperator\tr{tr}

\newcommand{\mres}{\mathbin{\vrule height 1.6ex depth 0pt width
0.13ex\vrule height 0.13ex depth 0pt width 1.3ex}}

\title[Nonlocal anisotropic interactions of Coulomb type]{Nonlocal anisotropic interactions of Coulomb type}

\author[M.G. Mora]{Maria Giovanna Mora} 

\address[M.G. Mora]{Dipartimento di Matematica, Universit\`a di Pavia, via Ferrata 5, 27100 Pavia (Italy)}
\email{mariagiovanna.mora@unipv.it}

\begin{document}

\begin{abstract} 
In this paper we review some recent results on nonlocal interaction problems.
The focus is on interaction kernels that are anisotropic variants of the classical Coulomb kernel. In other words, while preserving the same singularity at zero of the Coulomb kernel, 
they present preferred directions of interaction.
For kernels of this kind and general confinement we will prove existence and uniqueness of minimisers of the corresponding energy.
In the case of a quadratic confinement we will review a recent result by Carrillo \& Shu about the explicit characterisation of minimisers, and present a new proof, which has the advantage of being extendable to higher dimension.
In light of this result, we will re-examine some previous works motivated by applications to dislocation theory in materials science.
Finally, we will discuss some related results and open questions.\bigskip

\noindent\textbf{AMS 2010 Mathematics Subject Classification:}  31A15 (primary); 49K20 (secondary)\medskip

\noindent \textbf{Keywords:} nonlocal energy, potential theory, anisotropic interaction, Coulomb potential

\end{abstract}

\maketitle

\begin{section}{Introduction}

The general goal of this review paper is the study of the minimisation problem for an energy of the form
\begin{equation}\label{int:I}
I(\mu)=\iint_{\R^N\times\R^N}W(x-y)\,d\mu(y)d\mu(x)+\int_{\R^N}V(x)\,d\mu(x)
\end{equation}
defined for $\mu\in\P(\R^N)$. Here $\P(\R^N)$ stands for the space of probability measures in $\R^N$. In this formulation
a measure $\mu\in\P(\R^N)$ represents the distribution of a family of particles in $\R^N$, the first integral in $I(\mu)$ is called the {\em interaction energy},
whereas the last integral usually plays the role of a {\em confinement energy}.

Energies as $I$ arise as {\em mean-field limits} of discrete energies. More precisely, let us consider $n$ particles in $\R^N$ located at points
$x^1,x^2,\dots, x^n$ in $\R^N$ and let us define as their interaction energy the quantity
$$
\frac1{n^2}\sum_{j=1}^n\sum_{k\neq j} W(x^j-x^k).
$$
Since the order in which particles are considered is irrelevant, it is natural to assume the interaction kernel $W$ to be an even function.
This discrete energy describes {\em nonlocal} interactions in the sense that each particle interacts with any other particle in the system and not only with those in its immediate neighbourhood.
If one identifies the distribution of particles with the so-called {\em empirical measure}
$$
\frac1n \sum_{j=1}^n \delta_{x^j},
$$
one can show, under suitable assumptions of $W$ and $V$, that the discrete energies
$$
\frac1{n^2}\sum_{j=1}^n\sum_{k\neq j} W(x^j-x^k) + \frac1n \sum_{j=1}^nV(x^j)
$$
$\Gamma$-converge to $I$, as $n\to\infty$, with respect to the narrow convergence in $\P(\R^N)$ (see, e.g., \cite{Sca-thesis}). In other words, minimisers of $I$ describe the asymptotic behaviour of optimal distributions at the discrete level in the many-particle limit.

In many applications the typical interaction among particles is short-range repulsive and long-range attractive.
This behaviour can be reproduced in the energy $I$ by assuming:
\begin{itemize}

\item $W(x)\to+\infty$, as $x\to0$, so that the interaction energy blows up when particles get too close to one another;

\item $V(x)\to+\infty$ fast enough, as $|x|\to+\infty$, so that the confinement energy blows up when particles escape at infinity.

\end{itemize} 
As a model example, we can consider as $W$ the Coulomb kernel
$$
W(x)=\begin{cases}
-\log|x| & \text{ if } N=2, \smallskip\\
\dfrac1{|x|^{N-2}} & \text{ if } N\geq3,
\end{cases}
$$
and, as confinement potential, a power law $V(x)=|x|^p$ with $p>0$ or the indicator of a given compact set $K\subset\R^N$
$$
V(x)=\begin{cases}
0 & \text{ if } x\in K, \\
+\infty & \text{ if } x\not\in K.
\end{cases}
$$
In this last case (which we call {\em physical confinement}) minimising $I$ is equivalent to minimising the sole interaction energy on the class of probability measures supported in $K$.

Continuum energies as $I$, as well as their discrete counterparts, are relevant in a variety of applications, ranging from 
physics (electrostatics, Coulomb gases, Ginzburg-Landau theory) to biology (population dynamics)
and materials science. In particular, the Coulomb kernel is probably the most studied interaction kernel in physics and in mathematics.

Besides existence and uniqueness, one of the main questions in minimising $I$ is whether minimisers can be identified or at least some of their qualitative properties can be established. 
For instance, can we determine the dimension of their supports and their shape? Is the distribution ``regular'' on the support? One of the key difficulties in addressing these questions is
the nonlocal nature of the problem: given a distribution $\mu$, any local perturbation of $\mu$, however small, will have a global impact on the interaction energy. Moreover, numerical
simulations show that, according to the different choice of $W$ and $V$, minimisers may present a rich variety of geometries and shapes (see, e.g, \cite{KSUB}).

In this paper we will focus on the two-dimensional case $N=2$ and on interaction kernels of the form
\begin{equation}\label{int:Wk}
W(x)=-\log|x|+\kappa(x),
\end{equation}
where $\kappa$ is an even $0$-homogeneous function, smooth enough outside $0$.
The kernel $W$ can be seen as a perturbation (not small though) of the 2d Coulomb kernel. Since $\kappa$ is $0$-homogeneous,
$\kappa(x)$ depends only on the angle that $x$ forms with respect to a given reference axis. In this sense we call $\kappa$ an {\em anisotropic kernel}, meaning
that it introduces some preferred directions of interaction. 

The goal of this paper is to review the most recent results about existence, uniqueness, and characterisation of minimisers for kernels $W$ of the form \eqref{int:Wk} 
and general confinements $V$. The common thread of these results is the following key idea. A clever way to look at this class of problems is via Fourier analysis.
In fact, if we denote by $\widehat f$ the Fourier transform of $f$ given by
$$
\widehat f(\xi)=\frac1{2\pi}\int_{\R^2}f(x)e^{-i\xi\cdot x}\, dx \quad \text{ for }\xi\in\R^2,
$$
formally we have
$$
\widehat{W\ast\mu}=2\pi\,\widehat W\widehat\mu
$$
and by Plancherel Theorem
\begin{equation}\label{int:formal}
\iint_{\R^2\times\R^2}W(x-y)\,d\mu(y)d\mu(x) = \int_{\R^2}(W\ast\mu)\,d\mu=
\int_{\R^2}\widehat{W\ast\mu}(\xi)\overline{\widehat \mu}(\xi)\, d\xi=
2\pi \int_{\R^2}\widehat W|\widehat\mu|^2\, d\xi.
\end{equation}
In other words, the nonlocal interaction can be expressed in a local form in the Fourier space. Note, however, that \eqref{int:formal} holds true only under specific assumptions for $\widehat W$ and $\mu$ (see Proposition~\ref{prop:pars}). Using \eqref{int:formal} we will show that a sign condition on $\widehat W$ guarantees strict convexity of the energy and, therefore, uniqueness of minimisers (see Section~\ref{sec:uniq}). Moreover, the inversion formula for the Fourier transform will be a crucial ingredient in the characterisation results of Section~\ref{sec:ch}.

\subsection{Motivation}
The study of interaction kernels of the form \eqref{int:Wk} is motivated by materials science, more precisely, by {\em dislocation theory}.
Dislocations are defects in the crystalline lattice of a metal, whose presence and concerted movement favour plastic slips, that is, relative slips of atomic layers, that macroscopically result into a shearing plastic deformation.

Let us consider an idealised three-dimensional cubic lattice, where all two-dimensional sections along a certain direction are assumed to be identical.
In this simplified two-dimensional setting a dislocation of edge type looks as in Fig.~\ref{fig1}.
\begin{figure}
\begin{tikzpicture}
\draw[thick] (-0.5,5) -- (4.7,5);
\filldraw[color=blue] (0,5) circle (3pt);
\filldraw[color=blue] (0.7,5) circle (3pt);
\filldraw[color=blue] (1.4,5) circle (3pt);
\filldraw[color=blue] (2.1,5) circle (3pt);
\filldraw[color=blue] (2.8,5) circle (3pt);
\filldraw[color=blue] (3.5,5) circle (3pt);
\filldraw[color=blue] (4.2,5) circle (3pt);
\draw[thick] (-0.5,4.3) -- (4.7,4.3);
\filldraw[color=blue] (0,4.3) circle (3pt);
\filldraw[color=blue] (0.71,4.3) circle (3pt);
\filldraw[color=blue] (1.42,4.3) circle (3pt);
\filldraw[color=blue] (2.1,4.3) circle (3pt);
\filldraw[color=blue] (2.78,4.3) circle (3pt);
\filldraw[color=blue] (3.49,4.3) circle (3pt);
\filldraw[color=blue] (4.2,4.3) circle (3pt);
\draw[thick] (-0.45,3.6) -- (4.65,3.6);
\filldraw[color=blue] (0.05,3.6) circle (3pt);
\filldraw[color=blue] (0.76,3.6) circle (3pt);
\filldraw[color=blue] (1.44,3.6) circle (3pt);
\filldraw[color=blue] (2.1,3.6) circle (3pt);
\filldraw[color=blue] (2.76,3.6) circle (3pt);
\filldraw[color=blue] (3.44,3.6) circle (3pt);
\filldraw[color=blue] (4.15,3.6) circle (3pt);
\draw[thick] (-0.4,2.9) -- (4.6,2.9);
\filldraw[color=blue] (0.1,2.9) circle (3pt);
\filldraw[color=blue] (0.825,2.9) circle (3pt);
\filldraw[color=blue] (1.6,2.9) circle (3pt);
\filldraw[color=blue] (2.6,2.9) circle (3pt);
\filldraw[color=blue] (3.375,2.9) circle (3pt);
\filldraw[color=blue] (4.1,2.9) circle (3pt);
\draw[thick] (-0.35,2.2) -- (4.55,2.2);
\filldraw[color=blue] (0.15,2.2) circle (3pt);
\filldraw[color=blue] (0.875,2.2) circle (3pt);
\filldraw[color=blue] (1.675,2.2) circle (3pt);
\filldraw[color=blue] (2.525,2.2) circle (3pt);
\filldraw[color=blue] (3.325,2.2) circle (3pt);
\filldraw[color=blue] (4.05,2.2) circle (3pt);
\draw[thick] (-0.25,1.5) -- (4.45,1.5);
\filldraw[color=blue] (0.25,1.5) circle (3pt);
\filldraw[color=blue] (0.95,1.5) circle (3pt);
\filldraw[color=blue] (1.7,1.5) circle (3pt);
\filldraw[color=blue] (2.5,1.5) circle (3pt);
\filldraw[color=blue] (3.25,1.5) circle (3pt);
\filldraw[color=blue] (3.95,1.5) circle (3pt);
\draw[thick] (-0.15,0.8) -- (4.35,0.8);
\filldraw[color=blue] (0.35,0.8) circle (3pt);
\filldraw[color=blue] (1.05,0.8) circle (3pt);
\filldraw[color=blue] (1.75,0.8) circle (3pt);
\filldraw[color=blue] (2.45,0.8) circle (3pt);
\filldraw[color=blue] (3.15,0.8) circle (3pt);
\filldraw[color=blue] (3.85,0.8) circle (3pt);
\draw[thick] (0.35, 0.3) -- (0.35,0.8) -- (0.25,1.5) -- (0.15,2.2) -- (0.1,2.9) -- (0.05,3.6) -- (0,4.3) -- (0,5.5);
\draw[thick] (1.05, 0.3) -- (1.05,0.8) -- (0.95,1.5) -- (0.875,2.2) -- (0.825,2.9) -- (0.76,3.6) -- (0.71,4.3) -- (0.7,5) -- (0.7,5.5);
\draw[thick] (1.75, 0.3) -- (1.75,0.8) -- (1.7,1.5) -- (1.675,2.2) -- (1.6,2.9) -- (1.44,3.6) -- (1.42,4.3) -- (1.4, 5) -- (1.4,5.5);
\draw[thick] (2.1,3.6) -- (2.1,5.5);
\draw[thick] (2.45,0.3) -- (2.45,0.8) -- (2.5,1.5) --(2.525, 2.2) -- (2.6,2.9) -- (2.76,3.6) --(2.78,4.3) -- (2.8,5) -- (2.8,5.5);
\draw[thick] (3.15, 0.3) -- (3.15, 0.8) -- (3.25,1.5) -- (3.325, 2.2) -- (3.375,2.9) -- (3.44, 3.6) -- (3.49,4.3) -- (3.5,5) -- (3.5,5.5);
\draw[thick] (3.85, 0.3) -- (3.85, 0.8) -- (3.95,1.5) -- (4.05, 2.2) -- (4.1,2.9) -- (4.15,3.6) -- (4.2, 4.3) -- (4.2,5.5);
\end{tikzpicture}
\caption{A cubic crystal with a dislocation.}\label{fig1}
\end{figure}
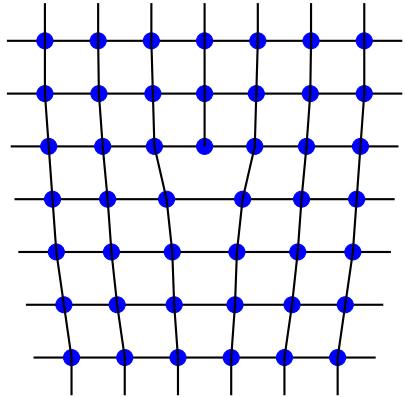

The presence of a dislocation can be detected as follows. One draws a so-called {\em Burgers circuit}, that is, a closed circuit
enclosing the defect. If we draw the same circuit in a perfect reference crystal, the circuit does not close up, see Fig.~\ref{fig2}. The vector that needs to be added to close the circuit is defined as the {\em Burgers vector}. The Burgers vector is thus a measure of the discrepancy between the distorted lattice and a perfect lattice.

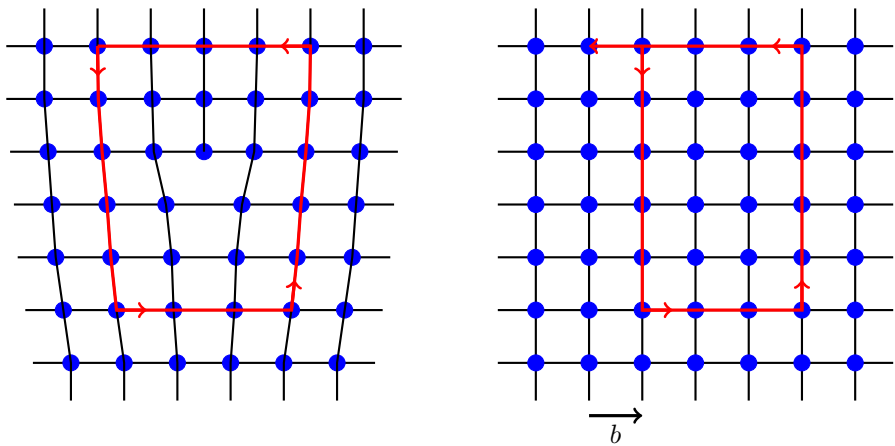
\begin{figure}
\begin{tikzpicture}
\draw[thick] (-0.5,5) -- (4.7,5);
\filldraw[color=blue] (0,5) circle (3pt);
\filldraw[color=blue] (0.7,5) circle (3pt);
\filldraw[color=blue] (1.4,5) circle (3pt);
\filldraw[color=blue] (2.1,5) circle (3pt);
\filldraw[color=blue] (2.8,5) circle (3pt);
\filldraw[color=blue] (3.5,5) circle (3pt);
\filldraw[color=blue] (4.2,5) circle (3pt);
\draw[thick] (-0.5,4.3) -- (4.7,4.3);
\filldraw[color=blue] (0,4.3) circle (3pt);
\filldraw[color=blue] (0.71,4.3) circle (3pt);
\filldraw[color=blue] (1.42,4.3) circle (3pt);
\filldraw[color=blue] (2.1,4.3) circle (3pt);
\filldraw[color=blue] (2.78,4.3) circle (3pt);
\filldraw[color=blue] (3.49,4.3) circle (3pt);
\filldraw[color=blue] (4.2,4.3) circle (3pt);
\draw[thick] (-0.45,3.6) -- (4.65,3.6);
\filldraw[color=blue] (0.05,3.6) circle (3pt);
\filldraw[color=blue] (0.76,3.6) circle (3pt);
\filldraw[color=blue] (1.44,3.6) circle (3pt);
\filldraw[color=blue] (2.1,3.6) circle (3pt);
\filldraw[color=blue] (2.76,3.6) circle (3pt);
\filldraw[color=blue] (3.44,3.6) circle (3pt);
\filldraw[color=blue] (4.15,3.6) circle (3pt);
\draw[thick] (-0.4,2.9) -- (4.6,2.9);
\filldraw[color=blue] (0.1,2.9) circle (3pt);
\filldraw[color=blue] (0.825,2.9) circle (3pt);
\filldraw[color=blue] (1.6,2.9) circle (3pt);
\filldraw[color=blue] (2.6,2.9) circle (3pt);
\filldraw[color=blue] (3.375,2.9) circle (3pt);
\filldraw[color=blue] (4.1,2.9) circle (3pt);
\draw[thick] (-0.35,2.2) -- (4.55,2.2);
\filldraw[color=blue] (0.15,2.2) circle (3pt);
\filldraw[color=blue] (0.875,2.2) circle (3pt);
\filldraw[color=blue] (1.675,2.2) circle (3pt);
\filldraw[color=blue] (2.525,2.2) circle (3pt);
\filldraw[color=blue] (3.325,2.2) circle (3pt);
\filldraw[color=blue] (4.05,2.2) circle (3pt);
\draw[thick] (-0.25,1.5) -- (4.45,1.5);
\filldraw[color=blue] (0.25,1.5) circle (3pt);
\filldraw[color=blue] (0.95,1.5) circle (3pt);
\filldraw[color=blue] (1.7,1.5) circle (3pt);
\filldraw[color=blue] (2.5,1.5) circle (3pt);
\filldraw[color=blue] (3.25,1.5) circle (3pt);
\filldraw[color=blue] (3.95,1.5) circle (3pt);
\draw[thick] (-0.15,0.8) -- (4.35,0.8);
\filldraw[color=blue] (0.35,0.8) circle (3pt);
\filldraw[color=blue] (1.05,0.8) circle (3pt);
\filldraw[color=blue] (1.75,0.8) circle (3pt);
\filldraw[color=blue] (2.45,0.8) circle (3pt);
\filldraw[color=blue] (3.15,0.8) circle (3pt);
\filldraw[color=blue] (3.85,0.8) circle (3pt);
\draw[thick] (0.35, 0.3) -- (0.35,0.8) -- (0.25,1.5) -- (0.15,2.2) -- (0.1,2.9) -- (0.05,3.6) -- (0,4.3) -- (0,5.5);
\draw[thick] (1.05, 0.3) -- (1.05,0.8) -- (0.95,1.5) -- (0.875,2.2) -- (0.825,2.9) -- (0.76,3.6) -- (0.71,4.3) -- (0.7,5) -- (0.7,5.5);
\draw[thick] (1.75, 0.3) -- (1.75,0.8) -- (1.7,1.5) -- (1.675,2.2) -- (1.6,2.9) -- (1.44,3.6) -- (1.42,4.3) -- (1.4, 5) -- (1.4,5.5);
\draw[thick] (2.1,3.6) -- (2.1,5.5);
\draw[thick] (2.45,0.3) -- (2.45,0.8) -- (2.5,1.5) --(2.525, 2.2) -- (2.6,2.9) -- (2.76,3.6) --(2.78,4.3) -- (2.8,5) -- (2.8,5.5);
\draw[thick] (3.15, 0.3) -- (3.15, 0.8) -- (3.25,1.5) -- (3.325, 2.2) -- (3.375,2.9) -- (3.44, 3.6) -- (3.49,4.3) -- (3.5,5) -- (3.5,5.5);
\draw[thick] (3.85, 0.3) -- (3.85, 0.8) -- (3.95,1.5) -- (4.05, 2.2) -- (4.1,2.9) -- (4.15,3.6) -- (4.2, 4.3) -- (4.2,5.5);
\draw[very thick, color=red] (0.7,5) -- (3.5,5) -- (3.49,4.3) -- (3.44, 3.6) -- (3.375,2.9) -- (3.325, 2.2) -- (3.25,1.5) -- (0.95,1.5) -- (0.875,2.2) -- (0.825,2.9) -- (0.76,3.6) -- (0.71,4.3) -- (0.7,5);
\draw[very thick, color=red, ->] (0.7,5) -- (0.7057,4.6);
\draw[very thick, color=red, ->] (0.95,1.5) -- (1.35,1.5);
\draw[very thick, color=red, ->] (3.25,1.5) -- (3.2928,1.9);
\draw[very thick, color=red, <-] (3.1,5) -- (3.5,5);
\phantom{\draw[very thick, ->] (0.7,0.1) -- (1.4,0.1);
\node[below] at (1.05,0.1) {$b$};}
\end{tikzpicture}
\hspace{1cm}
\begin{tikzpicture}
\draw[thick] (-0.5,5) -- (4.7,5);
\draw[thick] (-0.5,4.3) -- (4.7,4.3);
\draw[thick] (-0.5,3.6) -- (4.7,3.6);
\draw[thick] (-0.5,2.9) -- (4.7,2.9);
\draw[thick] (-0.5,2.2) -- (4.7,2.2);
\draw[thick] (-0.5,1.5) -- (4.7,1.5);
\draw[thick] (-0.5,0.8) -- (4.7,0.8);
\draw[thick] (0,0.3) -- (0,5.5);
\draw[thick] (0.7,0.3) -- (0.7,5.5);
\draw[thick] (1.4,0.3) -- (1.4,5.5);
\draw[thick] (2.1,0.3) -- (2.1,5.5);
\draw[thick] (2.8,0.3) -- (2.8,5.5);
\draw[thick] (3.5,0.3) -- (3.5,5.5);
\draw[thick] (4.2,0.3) -- (4.2,5.5);
\filldraw[color=blue] (0,5) circle (3pt);
\filldraw[color=blue] (0.7,5) circle (3pt);
\filldraw[color=blue] (1.4,5) circle (3pt);
\filldraw[color=blue] (2.1,5) circle (3pt);
\filldraw[color=blue] (2.8,5) circle (3pt);
\filldraw[color=blue] (3.5,5) circle (3pt);
\filldraw[color=blue] (4.2,5) circle (3pt);
\filldraw[color=blue] (0,4.3) circle (3pt);
\filldraw[color=blue] (0.7,4.3) circle (3pt);
\filldraw[color=blue] (1.4,4.3) circle (3pt);
\filldraw[color=blue] (2.1,4.3) circle (3pt);
\filldraw[color=blue] (2.8,4.3) circle (3pt);
\filldraw[color=blue] (3.5,4.3) circle (3pt);
\filldraw[color=blue] (4.2,4.3) circle (3pt);
\filldraw[color=blue] (0,3.6) circle (3pt);
\filldraw[color=blue] (0.7,3.6) circle (3pt);
\filldraw[color=blue] (1.4,3.6) circle (3pt);
\filldraw[color=blue] (2.1,3.6) circle (3pt);
\filldraw[color=blue] (2.8,3.6) circle (3pt);
\filldraw[color=blue] (3.5,3.6) circle (3pt);
\filldraw[color=blue] (4.2,3.6) circle (3pt);
\filldraw[color=blue] (0,2.9) circle (3pt);
\filldraw[color=blue] (0.7,2.9) circle (3pt);
\filldraw[color=blue] (1.4,2.9) circle (3pt);
\filldraw[color=blue] (2.1,2.9) circle (3pt);
\filldraw[color=blue] (2.8,2.9) circle (3pt);
\filldraw[color=blue] (3.5,2.9) circle (3pt);
\filldraw[color=blue] (4.2,2.9) circle (3pt);
\filldraw[color=blue] (0,2.2) circle (3pt);
\filldraw[color=blue] (0.7,2.2) circle (3pt);
\filldraw[color=blue] (1.4,2.2) circle (3pt);
\filldraw[color=blue] (2.1,2.2) circle (3pt);
\filldraw[color=blue] (2.8,2.2) circle (3pt);
\filldraw[color=blue] (3.5,2.2) circle (3pt);
\filldraw[color=blue] (4.2,2.2) circle (3pt);
\filldraw[color=blue] (0,1.5) circle (3pt);
\filldraw[color=blue] (0.7,1.5) circle (3pt);
\filldraw[color=blue] (1.4,1.5) circle (3pt);
\filldraw[color=blue] (2.1,1.5) circle (3pt);
\filldraw[color=blue] (2.8,1.5) circle (3pt);
\filldraw[color=blue] (3.5,1.5) circle (3pt);
\filldraw[color=blue] (4.2,1.5) circle (3pt);
\filldraw[color=blue] (0,0.8) circle (3pt);
\filldraw[color=blue] (0.7,0.8) circle (3pt);
\filldraw[color=blue] (1.4,0.8) circle (3pt);
\filldraw[color=blue] (2.1,0.8) circle (3pt);
\filldraw[color=blue] (2.8,0.8) circle (3pt);
\filldraw[color=blue] (3.5,0.8) circle (3pt);
\filldraw[color=blue] (4.2,0.8) circle (3pt);
\draw[very thick, color=red, ->] (1.4,5) -- (1.4,1.5) -- (3.5,1.5) -- (3.5,5) -- (0.7,5);
\draw[very thick, color=red, ->] (1.4,5) -- (1.4,4.6);
\draw[very thick, color=red, ->] (1.4,1.5) -- (1.8,1.5);
\draw[very thick, color=red, ->] (3.5,1.5) -- (3.5,1.9);
\draw[very thick, color=red, ->] (3.5,5) -- (3.1,5);
\draw[very thick, ->] (0.7,0.1) -- (1.4,0.1);
\node[below] at (1.05,0.1) {$b$};
\end{tikzpicture}
\caption{The Burgers circuit and the Burgers vector $b$.}\label{fig2}
\end{figure}

Assume now that every dislocation has the same Burgers vector (which for simplicity we set to be $e_1$).
We would like to compute the interaction force between two dislocations, the so-called {\em Peach-K\"ohler force}.
To make this computation we consider a semi-discrete setting, where the crystal is described as a continuum medium
(as if the discrete lattice structure were averaged out), but dislocations are still modelled as point singularities. 
In the framework of linear elasticity the fundamental strain generated by a dislocation of Burgers vector $e_1$ located at $0$ is a solution $\beta:\R^2\to\mtwo$ of the problem
$$
\begin{cases}
\Div\C\beta=0 & \text{ in }\R^2, \\
\curl\beta=\delta_0 e_1 & \text{ in }\R^2, 
\end{cases}
$$
where $\C$ is the tensor of linear elasticity and $\delta_0$ denotes the Dirac delta at $0$. The Peach-K\"ohler force can then be computed by the formula
$$
F=\C\beta e_1\times e_3,
$$
where $\C\beta$ is the stress associated with the fundamental strain, $e_1$ is the Burgers vector, and $e_3$ represents the dislocation line (in our simplified 2d setting the dislocation line is orthogonal to the plane of the 2d section). For $x\in\R^2$ the Peach-K\"ohler force $F(x)$ is the configurational force that a dislocation at $x$ experiences because of the dislocation at $0$. It turns out (see \cite[Chapter~13-4]{HL}) that
$$
F=-(c\nabla W,0)
$$
for some positive material constant $c>0$, where $W$ is of the form \eqref{int:Wk} with
\begin{equation}\label{kappa-aniso}
\kappa(x)=-\frac14 \frac{a+b}a\log\Big(\frac{x_1^2+(a+b)^2x_2^2}{|x|^2}\Big)+\frac14 \frac{b-a}a\log\Big(\frac{x_1^2+(b-a)^2x_2^2}{|x|^2}\Big)
\end{equation}
for $x=(x_1,x_2)\in\R^2$.
Here $b>a>0$ are material constants. If the medium is isotropic (which corresponds to $a\to0^+$, $b=1$), the above expression reduces, up to additive constants, to
\begin{equation}\label{kappa-iso}
\kappa(x)=\frac{x_1^2}{|x|^2}.
\end{equation}

We would like to predict the optimal distribution of dislocations at equilibrium in this setting.
For simplicity, 
let us focus on the isotropic case \eqref{kappa-iso} and assume to have exactly two dislocations located at $x$ and $y$. Their interaction energy is given by
$$
-\log|x-y|+\frac{(x_1-y_1)^2}{|x-y|^2}.
$$
The Coulomb term forces $x$ and $y$ to be as far as possible (this repulsive behaviour is counterbalanced by the presence of some confinement), whereas the anisotropic term is minimised when $x_1=y_1$, that is, when the two dislocations are aligned vertically.
Does the same phenomenon occur in the mesoscopic description \eqref{int:I}? In other words, let $I$ be the energy defined by
\begin{equation}\label{int:I2d}
I(\mu)=\iint_{\R^2\times\R^2}\Big(-\log|x-y|+\frac{(x_1-y_1)^2}{|x-y|^2}\Big)\,d\mu(y)d\mu(x)+\int_{\R^2}V(x)\,d\mu(x),
\end{equation}
where $\mu\in\P(\R^2)$ represents now the distribution of a family of dislocations of Burgers vector $e_1$. Is it true that minimisers of $I$ have a 1d vertical support? And if so, what is their distribution? This last question is actually not difficult to answer. In fact, let $\mu\in\P(\R^2)$ be a measure of the form $\mu=\delta_{\bar x_1}\otimes\nu$ with $\bar x_1\in\R$ and $\nu\in\P(\R)$. When restricted to measures of this form the energy $I$ reduces to
$$
J(\bar x_1,\nu):=-\iint_{\R\times\R}\log|s-t|\,d\nu(t)d\nu(s)+\int_{\R}V(\bar x_1,t)\,d\nu(t).
$$
Thus, if the minimiser of $I$ has a vertical support, its vertical projection minimises the 1d functional $J(\bar x_1, \cdot)$ among all measures in $\P(\R)$.
The functional $J$ is known in the literature as the log-gas energy and its minimisers can be explicitly computed for several confinements. For instance, for $V(x)=|x|^2$ 
the unique minimiser of $J(\bar x_1, \cdot)$ is the so-called {\em semicircle law}, that is, the measure
$$
\frac1\pi \sqrt{2-t^2}\LL^1\mres[-\sqrt2,\sqrt2](t)
$$
(and the optimal choice for $\bar x_1$ is clearly $\bar x_1=0$).
Therefore, in the case $V(x)=|x|^2$, if the energy $I$ in \eqref{int:I2d} has a minimiser $\mu$ with a vertical support, then necessarily $\mu$ is given by the semicircle law on the vertical axis
\begin{equation}\label{semicircle}
\frac1\pi \delta_0(x_1)\otimes\sqrt{2-x_2^2}\HH^1\mres[-\sqrt2,\sqrt2](x_2).
\end{equation}

\subsection{Overview of the results}
We now briefly review the main results about explicit characterisation of minimisers for kernels of the form \eqref{int:Wk}.
We consider an energy of the form \eqref{int:I2d} with confinement $V(x)=|x|^2$ and we introduce a parameter $\alpha\in\R$ in front of the anisotropic term.
The role of $\alpha$ is that of tuning the strength of the anisotropic interaction. 
The energy 
\begin{equation}\label{int:Ialfa}
I_\alpha(\mu)=\iint_{\R^2\times\R^2}\Big(-\log|x-y|+\alpha\frac{(x_1-y_1)^2}{|x-y|^2}\Big)\,d\mu(y)d\mu(x)+\int_{\R^2}|x|^2\,d\mu(x)
\end{equation}
has a unique minimiser, which can be characterised as follows:
\begin{itemize}
\item if $\alpha=0$ (purely Coulomb case), the minimiser is the so-called {\em circle law} 
$$
\frac1\pi\chi_{B_1(0)};
$$
the derivation of this classical result is attributed to Ginibre \cite{G}, Mehta \cite{M}, and Girko~\cite{Gk};

\item if $\alpha=1$ (dislocation case), the minimiser is the semicircle law on the vertical axis, that is, the measure \eqref{semicircle};
this was proved in \cite{MRS};

\item if $\alpha\in(0,1)$, the minimiser is given by the {\em ellipse law}
\begin{equation}\label{E-alfa}
\frac1{|E_\alpha|}\chi_{E_\alpha},
\end{equation}
where 
$$
E_\alpha=\Big\{ x\in\R^2: \ \frac{x_1^2}{1-\alpha}+\frac{x_2^2}{1+\alpha}\leq 1\Big\};
$$
this result is contained in \cite{CMMRSV}.

\end{itemize}
We mention that the minimiser of $I_\alpha$ can be actually identified for any value of $\alpha\in\R$, see Theorem~\ref{thm:disl}.

The above results show that at $\alpha=1$ the anisotropy has a dramatic effect on the structure of the minimiser and, in particular,
on its dimensionality. For any $\alpha\in[0,1)$ the minimiser is given by
a uniform distribution on a two-dimensional set and it is only at $\alpha=1$ that an abrupt loss of dimensionality occurs.
A natural question is whether this phenomenon can be explained in mathematical terms and how much of
this analysis is bound to the specific choice of the anisotropy.

In \cite{MMRSV21-NA} we conjectured that the loss of dimensionality could be related to a change of sign of the Fourier transform of the interaction kernel.
Indeed, we have that
$$
\widehat{W_\alpha}(\xi)= c_\alpha\delta_0+\frac{(1-\alpha)\xi_1^2+(1+\alpha)\xi_2^2}{|\xi|^4},
$$
see \eqref{Fourier Wa} below. Thus, for $\alpha\in[0,1)$ the Fourier transform of $W_\alpha$ is strictly positive outside 0, whereas it is only non-negative for $\alpha=1$.

Let us consider now an interaction kernel of the form \eqref{int:Wk} with a general even and $0$-homoge\-ne\-ous anisotropy $\kappa$, and confinement $V(x)=|x|^2$.
In \cite{MMRSV21} it was proved that, if $\kappa$ is small enough in $C^3(\mathbb S^1)$, then the minimiser is unique and given by the normalised characteristic function
of the domain enclosed by an ellipse centered at the origin. Note that the smallness assumption on $\kappa$ implies in particular that the Fourier transform of $W$ outside 0 is strictly positive.

Recently, Carrillo \& Shu proved that the smallness assumption in \cite{MMRSV21} can be lifted, as long as the condition on the non-negativity of the Fourier transform of $W$ outside 0 is preserved.
More precisely, in \cite{CS22} the following remarkable result was proved: let $W$ be an interaction kernel of the form \eqref{int:Wk} with $\kappa$ even, $0$-homogeneous, and smooth enough on $\mathbb S^1$, and let the confinement be $V(x)=|x|^2$; then,
\begin{itemize}
\item if $\widehat W>0$ outside 0, the unique minimiser is the normalised characteristic function
of the domain enclosed by an ellipse centered at the origin;
\item if $\widehat W\geq0$ outside 0, the unique minimiser is either as above or is a semicircle law on a line passing through the origin.
\end{itemize}
This result sheds some light on the relation between loss of dimensionality
and change of sign of the Fourier transform: for this class of kernels loss of dimensionality cannot occur as long as the Fourier transform is `not degenerate'.
However, explicit examples show that for a degenerate Fourier transform both cases (ellipse or semicircle law) can indeed occur, see Remark~\ref{ross}.

In this article we will first review the existence and uniqueness results for interaction kernels of the form \eqref{int:Wk} and general confinements.
In the case of the quadratic confinement $V(x)=|x|^2$ we will present a new proof of the result \cite{CS22} by Carrillo \& Shu, which follows the approach proposed in \cite{MMRSV22p} for the analogous problem in 3d. This different approach has the advantage of being extendable to higher dimension.
We will then deduce the characterisation of minimisers in the dislocation case \eqref{kappa-iso} from this general result. Finally, we will examine an example with physical confinement given by the domain enclosed by an ellipse and discuss some open questions.

\end{section}

\begin{section}{The existence result and first properties of minimisers}

Throughout the paper we will consider an interaction kernel of the form
\begin{equation}\label{kernel}
W(x)=-\log|x|+\kappa(x) \quad \text{ for } x\neq0, \qquad W(0)=+\infty,
\end{equation}
where $\kappa$ is an even $0$-homogeneous function of class $H^s$ on $\mathbb S^1$ with $s>3/2$ .
We will denote the 2d Coulomb kernel by $W_0$, that is,
$$
W_0(x)=-\log|x| \quad \text{ for } x\neq0, \qquad W_0(0)=+\infty.
$$
Since $\kappa$ is bounded on $\R^2\setminus\{0\}$, there exist two constants $C_1,C_2\in\R$ such that
\begin{equation}\label{bddW}
W_0(x)+C_1\leq W(x) \leq W_0(x)+C_2 \quad \text{ for every } x\in\R^2.
\end{equation}

\subsection{Logarithmic capacity}
For any compact set $K\subset\R^2$ we define the logarithmic capacity of $K$ as
$$
\cpc(K)=\Phi\Big(
\inf_{\mu\in\P(K)} \iint_{K\times K}W_0(x-y)\,d\mu(x)d\mu(y)\Big),
$$
where $\Phi(t)=e^{-t}$ for $t\in\R$ and $\Phi(+\infty)=0$. Note that the integral above is well defined (possibly equal to $+\infty$), since the integrand is bounded from below 
on $K\times K$.

If $B\subset\R^2$ is a Borel set, we define its capacity as
$$
\cpc(B)=\sup\big\{ \cpc(K): \ K\text{ compact, } K\subset B\big\}.
$$
We will say that a property holds {\it quasi everywhere} (q.e.) if it holds up to sets of zero capacity.

The key property of capacity is the following: if $\mu\in\P(\R^2)$ has compact support and satisfies
\begin{equation}\label{key-cap}
\iint_{\R^2\times\R^2}W(x-y)\,d\mu(x)d\mu(y)<+\infty,
\end{equation}
then $\mu(B)=0$ for every Borel set $B$ with $\cpc(B)=0$. In other words, if a property holds q.e., then it holds $\mu$-a.e.\ for any $\mu\in\P(\R^2)$ with compact support
and satisfying \eqref{key-cap} (as we will see, these are the relevant measures for the minimisation problem under study). This can be shown as follows: assume by contradiction that $\mu(B)>0$,
where $B\subset\R^2$ is a Borel set with $\cpc(B)=0$. Then there exists a compact set $K\subset B$ such that $\mu(K)>0$. 
Since $\mu$ is compactly supported, there exists $C_0>0$ such that
$$
W(x-y)\geq-C_0 \quad \text{ for every } (x,y)\in (\supp\mu)^2.
$$
Therefore, setting
$$
\nu:=\frac{1}{\mu(K)}\mu\mres K,
$$
by \eqref{key-cap} we have
$$
\iint_{K\times K}W(x-y)\,d\nu(x)d\nu(y)\leq \frac1{\mu(K)^2}\iint_{\R^2\times \R^2}W(x-y)\,d\mu(x)d\mu(y)+\frac{C_0}{\mu(K)^2}-C_0<+\infty.
$$
On the other hand, from \eqref{bddW} it follows that
$$
\iint_{K\times K}W(x-y)\,d\nu(x)d\nu(y)\geq \iint_{K\times K}W_0(x-y)\,d\nu(x)d\nu(y) +C_1.
$$
This would imply that $\cpc(K)>0$, contradicting the assumption $\cpc(B)=0$.

In a similar way one can show that a countable union of sets with zero capacity has zero capacity.

\subsection{The confinement potential}\label{ss:conf}
In this section we assume the confinement potential $V:\R^2\to\R\cup\{+\infty\}$ to be a lower semicontinuous function, which is bounded from below
and satisfies the following conditions:
\begin{equation}\label{V1}
\lim_{|x|\to+\infty}\Big(\frac12 V(x)-\log|x|\Big)=+\infty
\end{equation}
and
\begin{equation}\label{V2}
\cpc\big(\{x\in\R^2: \ V(x)<+\infty\}\big)>0.
\end{equation}
Examples of admissible confinements are the power laws $V(x)=|x|^p$ with $p>0$ or the indicator function of a compact set of positive capacity.

\subsection{Main results}
For every $\mu\in\P(\R^2)$ we define
\begin{equation}\label{realdef}
I(\mu)=\iint_{\R^2\times\R^2}\Big(W(x-y)+\frac12 V(x)+\frac12 V(y)\Big)\,d\mu(y)d\mu(x),
\end{equation}
where the interaction kernel $W$ is as in \eqref{kernel} and the confinement potential $V$ as in Section~\ref{ss:conf}.

\begin{theorem}[\cite{Fro,SaTo}]\label{thm:exist}
The energy $I$ is well defined and has a minimiser in $\P(\R^2)$. If $\mu$ is a minimiser of~$I$, then $I(\mu)<+\infty$, the support of $\mu$ is a compact set, and
$\mu$ satisfies the following Euler-Lagrange equations: there exists a constant $c=c(\mu)$ such that
\begin{equation}\label{EL1}
(W\ast\mu)(x)+\frac12 V(x)=c \quad \text{ for }\mu\text{-a.e.\ } x\in \supp\mu,
\end{equation}
and
\begin{equation}\label{EL2}
(W\ast\mu)(x)+\frac12 V(x)\geq c \quad \text{for q.e.\ } x\in\R^2.
\end{equation}
\end{theorem}

\begin{proof}
The function $(x,y)\mapsto W(x-y)+\frac12 V(x)+\frac12 V(y)$ is lower semicontinuous and blows up at infinity by \eqref{V1}, hence it is bounded from below by a constant $-c_1$ with $c_1>0$.
Therefore, the energy $I$ in \eqref{realdef} is well defined (possibly equal to $+\infty$) and $\inf I>-\infty$. Note also that the two representations \eqref{int:I}
and \eqref{V1} coincide whenever the interaction energy is well defined and not equal to $-\infty$.

Assumption \eqref{V2} guarantees that $\inf I<+\infty$. Indeed, writing $\{V<+\infty\}$ as the union over $n\in\N$ of the compact sets $\{V\leq n\}$, we must have
$\cpc\big(\{V\leq n_0\}\big)>0$ for some $n_0\in\N$, that is, there exists a probability measure $\mu_0$ with support in $\{V\leq n_0\}$ such that
$$
\iint_{\R^2\times\R^2}W_0(x-y)\,d\mu_0(x)d\mu_0(y)<+\infty.
$$
By \eqref{bddW} this implies that
$$
\iint_{\R^2\times\R^2}W(x-y)\,d\mu_0(x)d\mu_0(y)<+\infty.
$$
On the other hand,
$$
\int_{\R^2}V(x)\,d\mu_0(x)\leq n_0,
$$
hence $I(\mu_0)<+\infty$.

Existence of a minimiser follows by the Direct Method of the Calculus of Variations. Let $(\mu_n)_n$ be a minimising sequence. Since
$\inf I<+\infty$, there exists a constant $C>0$ such that $I(\mu_n)\leq C$ for every $n\in\N$. By 
\eqref{V1} for every $M>0$ there exists a compact set $K\subset\R^2$ such that
$$
W(x-y)+\frac12 V(x)+\frac12 V(y)\geq M \quad \text{ for } (x,y)\not\in K\times K.
$$
Therefore, for every $n\in N$
\begin{eqnarray*}
C \ \geq \ I(\mu_n) & \geq & M(\mu_n\otimes\mu_n)\big((K\times K)^c\big) -c_1\big(\mu_n(K)\big)^2
\\
& \geq & M\big(1- \big(\mu_n(K)\big)^2\big) -c_1
\\
& \geq & M \mu_n(K^c) -c_1.
\end{eqnarray*}
This inequality implies that the sequence $(\mu_n)_n$ is tight, hence, up to subsequences, $(\mu_n)_n$ converges narrowly to some $\mu\in\P(\R^2)$ (we refer to \cite[Chapter~5]{AGS} for the definition of tightness, narrow convergence, and their properties).
Since the integrand in $I$ is lower semicontinuous and bounded from below, we have
$$
I(\mu)\leq\liminf_{n\to\infty} I(\mu_n),
$$
hence $\mu$ is a minimiser.

Let now $\mu$ be a minimiser. In particular, $I(\mu)<+\infty$. By 
\eqref{V1} there exists a compact set $K\subset\R^2$ such that
\begin{equation}\label{bd-cpt}
W(x-y)+\frac12 V(x)+\frac12 V(y)\geq I(\mu)+1 \quad \text{ for } (x,y)\not\in K\times K.
\end{equation}
By taking $K$ larger if needed, we can assume that $\mu(K)>0$. We claim that $\supp\mu\subset K$.
Assume by contradiction that $\mu(K)<1$ and define
$$
\tilde\mu:=\frac{1}{\mu(K)}\mu\mres K \in\P(\R^2).
$$
By \eqref{bd-cpt} we deduce that
\begin{eqnarray*}
I(\tilde\mu) & = & \frac{1}{\big(\mu(K)\big)^2}\left(I(\mu) - \iint_{(K\times K)^c}\Big(W(x-y)+\frac12 V(x)+\frac12 V(y)\Big)\,d\mu(y)d\mu(x)
\right)
\\
& \leq & \frac{1}{\big(\mu(K)\big)^2} \Big( I(\mu) -(I(\mu)+1)(1-\big(\mu(K)\big)^2) \Big)
\\
& = & I(\mu) +1-\frac1{\big(\mu(K)\big)^2} <I(\mu),
\end{eqnarray*}
where the last inequality follows from the fact that $\mu(K)<1$. This contradicts the minimality of~$\mu$.

We conclude by showing that $\mu$ satisfies \eqref{EL1}--\eqref{EL2}. Let $\nu\in\P(\R^2)$ be a competitor such that its support is compact and $I(\nu)<+\infty$.
For $\e\in(0,1)$ we have that $(1-\e)\mu+\e\nu\in\P(\R^2)$. Therefore, by minimality
$$
I(\mu) \leq I((1-\e)\mu+\e\nu).
$$
Expanding the energy at the right-hand side yields
$$
0\leq -2\e\int_{\R^2} W\ast\mu\, d\mu +2\e\int_{\R^2} W\ast\mu\, d\nu +
\e\int_{\R^2}V\,d(\nu-\mu)+O(\e^2).
$$
Dividing by $2\e$ and sending $\e\to0^+$ lead to the following inequality:
\begin{equation}\label{EL2int}
\int_{\R^2} \Big(W\ast\mu+\frac12 V\Big)\, d\nu\geq \int_{\R^2} \Big(W\ast\mu+\frac12V\Big)\, d\mu =:c
\end{equation}
for every $\nu\in\P(\R^2)$ with compact support and $I(\nu)<+\infty$.

Set $P:=W\ast\mu+\frac12 V$ and assume by contradiction that
$$
\cpc\big(\big\{x\in\R^2: \ P(x)<c\big\}\big)>0.
$$
Note that by Fatou's lemma $P$ is lower semicontinuous, thus $\{P<c\}$ is a Borel set. By definition of capacity there exists a compact set $K\subset\R^2$ of positive capacity such that $P(x)<c$ for every $x\in K$. Therefore, there exists $\nu\in\P(K)$ such that
$$
\iint_{\R^2\times\R^2}W(x-y)\,d\nu(x)d\nu(y)<+\infty.
$$
Moreover, the confinement energy of $\nu$ is also finite, since
\begin{equation}\label{EL2int-f}
\int_{\R^2}\frac12V\, d\nu  =  \int_K \frac12V\, d\nu
< \int_K (c-W\ast\mu)\, d\nu
= c-\int_K W\ast\mu\, d\nu 
\end{equation}
and the right-hand side is finite by the bound from below of $W$ on compact sets.
Having finite energy and compact support, the measure $\nu$ has to satisfy condition \eqref{EL2int}. However, \eqref{EL2int-f} contradicts \eqref{EL2int}.
This proves \eqref{EL2}.

To prove \eqref{EL1} we note that by \eqref{EL2} we have $W\ast\mu+\frac12V\geq c$ q.e., hence $\mu$-a.e..
On the other hand, $c$ is by definition the integral of $W\ast\mu+\frac12V$ with respect to $\mu$, so necessarily \eqref{EL1} holds true.
\end{proof}

\begin{remark}
The previous result applies to much more general interaction kernels. Note in particular that the only properties of $\kappa$ we used are boundedness and lower semicontinuity.
\end{remark}

\end{section}

\begin{section}{The uniqueness result}\label{sec:uniq}

As mentioned in the introduction, the key assumption to guarantee uniqueness of minimiser is the sign condition $\widehat W\geq0$ outside $0$.
In fact, this assumption implies the strict convexity of the energy (when restricted to a suitable class). 
This result bears some similarities to Bochner Theorem, which characterises functions of positive type 
as those whose Fourier transform is a positive finite measure, see, e.g., \cite{Wil}. Functions of positive type provide in fact interaction kernels whose corresponding energies are convex on discrete measures. However, the regularity of $W$ and $\widehat W$ does not allow for a direct application of this theorem in our context.

Once strict convexity is established, not only the minimiser is unique, but
the two Euler-Lagrange conditions \eqref{EL1}--\eqref{EL2} are equivalent to minimality. Thus, identifying the unique minimiser reduces to finding a measure satisfying 
\eqref{EL1}--\eqref{EL2}. In order to do so it is essential to have an expression for $W\ast\mu$, in particular inside the support of $\mu$ where we need to verify \eqref{EL1}.
As we will see, the main idea is to rely again on Fourier analysis applying the inversion formula for the Fourier transform. 

\subsection{The Fourier transform}
We denote the Schwartz space of rapidly decreasing functions by $\S$ and its dual space by $\S'$, the so-called space of tempered distributions.
We recall that for every $\varphi\in \S$ its Fourier transform $\widehat\varphi\in\S$ is defined as
$$
\widehat \varphi(\xi)=\frac1{2\pi}\int_{\R^2}\varphi(x)e^{-i\xi\cdot x}\, dx \quad \text{ for }\xi\in\R^2.
$$
The map $\varphi\mapsto \widehat \varphi$ is a continuous linear isomorphism (with continuous inverse) of $\S$ into itself.
This allows one to extend the definition of Fourier transform to elements of $\S'$ simply by duality, that is, the Fourier transform $\widehat u$ of an element $u\in\S'$
is defined as
$$
\langle\, \widehat u, \varphi \rangle= \langle u, \widehat \varphi\, \rangle\quad \text{ for every }\varphi\in\S.
$$
We recall that the Fourier transform of the tempered distribution $\delta_0$ (the Dirac delta at $0$) is the constant $1/(2\pi)$.
Moreover, for every $\varphi\in \S$ and $j=1,2$ we have
$$
\widehat{\partial_{x_j}\varphi}(\xi)=i\xi_j\widehat\varphi(\xi) \quad \text{ for }\xi\in\R^2.
$$
 
\subsection{The Fourier transform of the kernel $W$} 
Both $W$ and $W_0$ are locally integrable functions with sublinear growth at infinity, hence they are tempered distributions.
Thus, their Fourier transforms are well defined as tempered distributions.
We start by computing the Fourier transform of the Coulomb kernel $W_0$. 
Since $\Delta W_0=-2\pi\delta_0$, formally we have
$$
-|\xi|^2\widehat{W_0}(\xi)=\widehat{\Delta W_0}(\xi)=-2\pi\widehat{\delta_0}(\xi)=-1.
$$ 
Hence, we would expect $\widehat{W_0}(\xi)=1/|\xi|^2$. However, this function is not integrable at $0$ and so, it does not define a tempered distribution.
In fact, the correct expression of $\widehat{W_0}$ is the following:
\begin{equation}\label{hatlog}
\langle \widehat{W_0}, \varphi\rangle =  c_0\varphi(0) +\int_{|\xi|\leq1}\frac1{|\xi|^2}(\varphi(\xi)-\varphi(0))\, d\xi
+\int_{|\xi|>1}\frac1{|\xi|^2}\varphi(\xi)\, d\xi
\end{equation}
for every $\varphi\in\S$. Here $c_0=\frac1{2\pi}(\gamma+\log\pi)$, where $\gamma$ is the Euler constant. For the proof of this formula we refer to \cite{Fol} or \cite{MRS}.

To compute the Fourier transform of the anisotropic kernel $\kappa$, it is convenient to pass to complex variables to simplify notation.
We replace $x$ by $z\in\C$ and write $z=|z|e^{i\theta}$ with $\theta\in[0,2\pi]$.
Since $\kappa$ is $0$-homogeneous, we have that $\kappa(z)=\kappa(e^{i\theta})$ and thus, it can be written as a Fourier series for $\theta\in[0,2\pi]$.
Moreover, since $\kappa$ is even, its Fourier series contains only the even terms, that~is,
$$
\kappa(e^{i\theta})=a_0+\sum_{n=1}^\infty \big( a_{2n}\cos(2n\theta) + b_{2n}\sin(2n\theta) \big)
$$
with $(a_{2n})_{n\in\N}, (b_{2n})_{n\in\N}\in\ell^2$. Note that for $n\in\N$
$$
\cos(n\theta)=\text{Re}\frac{z^n}{|z|^n}=:\frac{\phi_n(z)}{|z|^n}
\quad \text{ and } \quad \sin(n\theta)=\text{Im}\frac{z^n}{|z|^n}=:\frac{\psi_n(z)}{|z|^n}.
$$
The functions $\phi_n,\psi_n$ are harmonic homogeneous polynomials of order $n$ (in fact, they correspond to the so-called spherical harmonics in dimension $2$).
We can thus rewrite
\begin{equation}\label{Fourier-kappa}
\kappa(z)=a_0+\sum_{n=1}^\infty \Big( a_{2n}\frac{\phi_{2n}(z)}{|z|^{2n}} + b_{2n}\frac{\psi_{2n}(z)}{|z|^{2n}} \Big).
\end{equation}
Without loss of generality we can assume that $a_0=0$. In fact, adding a constant to $\kappa$ does not affect the minimisation problem under study.
The expression \eqref{Fourier-kappa} is particularly convenient to compute the Fourier transform of $\kappa$, owing to the following result, whose proof can be found in~\cite{S}.

\begin{lemma}\label{lemmaStein}
Let $\phi$ be a harmonic homogeneous polynomial of degree $m\geq1$ in $\R^2$.
Then the Fourier transform of $\phi(x)/|x|^m$ is given by 
\begin{equation}\label{stein}
\gamma_m\frac{\phi(\xi)}{|\xi|^{m+2}}
\end{equation}
for a suitable constant $\gamma_m$. For $m$ even, $m=2n$, one has $\gamma_{2n}=(-1)^n2n$.
\end{lemma}

Note that $\phi(\xi)/|\xi|^{m+2}$ behaves as $1/|\xi|^2$ for $\xi$ close to zero, therefore it is not integrable at $0$. Formula \eqref{stein} has to be interpreted as in \eqref{hatlog}, that is,
$$
\langle \gamma_m\frac{\phi}{|\cdot|^{m+2}}, \varphi\rangle =  \gamma_m\int_{|\xi|\leq1}\frac{\phi(\xi)}{|\xi|^{m+2}}(\varphi(\xi)-\varphi(0))\, d\xi
+\gamma_m \int_{|\xi|>1}\frac{\phi(\xi)}{|\xi|^{m+2}}\varphi(\xi)\, d\xi
$$
for every $\varphi\in\S$. 

By Lemma~\ref{lemmaStein} and \eqref{Fourier-kappa} with $a_0=0$ we deduce that
\begin{equation}\label{hatk}
\widehat\kappa(\xi)= \sum_{n=1}^\infty \Big( (-1)^n2na_{2n}\frac{\phi_{2n}(\xi)}{|\xi|^{2n+2}} + (-1)^n2nb_{2n}\frac{\psi_{2n}(\xi)}{|\xi|^{2n+2}} \Big).
\end{equation}
The argument can be made rigorous if $(2na_{2n})_{n\in\N}, (2nb_{2n})_{n\in\N}\in\ell^2$, that is, if $\kappa\in H^1(\mathbb S^1)$.

By \eqref{hatlog} and \eqref{hatk} we can write 
\begin{equation}\label{def:Omega}
\widehat W(\xi)=c_0\delta_0+\frac1{|\xi|^2}+\frac{\widehat\kappa(\xi/|\xi|)}{|\xi|^2}=: c_0\delta_0+\frac{\Om(\xi)}{|\xi|^2},
\end{equation}
where, with an abuse of notation, $\Om$ denotes the ``angular part'' of $\widehat W$. Note that $\Om$ is even and $0$-homogeneous, and that the above formula has to be interpreted as in \eqref{hatlog}.
In the following we will need $\Om\in C^0(\mathbb S^1)$. By Sobolev embedding this is true if $\widehat\kappa\in H^p(\mathbb S^1)$
with $p>1/2$. This is guaranteed by our assumption $\kappa\in H^s(\mathbb S^1)$ with $s>3/2$, owing to \eqref{hatk}. 

In the following we will refer to $\Om$ as the even and $0$-homogeneous function given by formula \eqref{def:Omega}.
Note that the purely Coulomb case corresponds to $\Om\equiv1$. Moreover, by \eqref{hatk} and \eqref{def:Omega} we have
\begin{equation}\label{int-Omega}
\frac1{2\pi}\int_{\mathbb S^1} \Om(\xi)\, d\HH^1(\xi)= \frac1{2\pi}\int_{\mathbb S^1}  \big( 1+\widehat\kappa(\xi) \big)\, d\HH^1(\xi)= 1.
\end{equation}

\begin{prop}\label{prop:pars}
Assume $\Om\geq0$ on $\mathbb S^1$. Let $\mu_0,\mu_1\in\P(\R^2)$ be two measures with compact support and finite interaction energy.
Let $\nu:=\mu_0-\mu_1$. Then
\begin{equation}\label{pars}
\int_{\R^2}W\ast\nu\, d\nu =2\pi\int_{\R^2}\frac{\Om(\xi)}{|\xi|^2}|\widehat\nu(\xi)|^2\, d\xi.
\end{equation}
In particular, the left-hand side is non-negative and
is equal to zero if and only if $\nu=0$, that is, if $\mu_0=\mu_1$.
\end{prop}


An immediate consequence of Proposition~\ref{prop:pars} is the following: if $\Om\geq0$ on $\mathbb S^1$,
then $I$ is strictly convex on the class of measures with compact support and finite interaction energy.
Indeed, let $\mu_0,\mu_1\in\P(\R^2)$ be two such measures with $\mu_0\neq\mu_1$. By Proposition~\ref{prop:pars}
we have that
$$
\int_{\R^2}W\ast(\mu_0-\mu_1)\, d(\mu_0-\mu_1)>0,
$$
hence
$$
\int_{\R^2}W\ast\mu_0\, d\mu_0 +\int_{\R^2}W\ast\mu_1\, d\mu_1 > 2 \int_{\R^2}W\ast\mu_0\, d\mu_1.
$$
Let now $\mu_t:=(1-t)\mu_0+t\mu_1$ with $t\in(0,1)$. From the above inequality we conclude that
\begin{eqnarray*}
\int_{\R^2}W\ast\mu_t\, d\mu_t & < & (1-t)^2\int_{\R^2}W\ast\mu_0\, d\mu_0 +t^2 \int_{\R^2}W\ast\mu_1\, d\mu_1
\\
& & {}+t(1-t) \Big(\int_{\R^2}W\ast\mu_0\, d\mu_0 +\int_{\R^2}W\ast\mu_1\, d\mu_1\Big)
\\
& = & (1-t)\int_{\R^2}W\ast\mu_0\, d\mu_0 +t \int_{\R^2}W\ast\mu_1\, d\mu_1.
\end{eqnarray*}

Since minimisers have compact support and finite interaction energy by Theorem~\ref{thm:exist}, this convexity property is enough to guarantee uniqueness and the equivalence of \eqref{EL1}--\eqref{EL2} with minimality.

\begin{proof}[Proof of Proposition~\ref{prop:pars}]
The heuristic idea to prove \eqref{pars} is to apply Plancherel Theorem, as we did formally in \eqref{int:formal}, and write
$$
\int_{\R^2}W\ast\nu\, d\nu =2\pi\int_{\R^2}\widehat W(\xi)|\widehat\nu(\xi)|^2\, d\xi.
$$
Taking into account \eqref{def:Omega} and the fact that
$$
\widehat\nu(0)=\frac1{2\pi}\int_{\R^2}\,d\nu=0,
$$
the right-hand side reduces to 
$$
2\pi\int_{\R^2}\widehat W(\xi)|\widehat\nu(\xi)|^2\, d\xi = 2\pi\int_{\R^2}\frac{\Om(\xi)}{|\xi|^2}|\widehat\nu(\xi)|^2\, d\xi.
$$
However, this argument is only formal, since we do not have enough regularity to apply Plancherel Theorem.
Our strategy is to prove \eqref{pars} by approximation. This is a rather delicate argument, since both sides of \eqref{pars} may a priori be not finite
and neither $W$ nor $\nu$ have a sign.

For $\e>0$ let $\varphi_\e$ be a radial mollifier supported on $B_\e(0)$. Let $\nu_\e:=\nu\ast\varphi_\e\in C^\infty_c(\R^2)\subset\S$.
Thus, $\widehat{\nu_\e}$ belongs to $\S$ and, in particular, to $L^\infty(\R^2)$. We note that $W\ast\nu_\e\in C^\infty(\R^2)$ and, since $\widehat{\nu_\e}$ is smooth
and $\widehat{\nu_\e}(0)=0$, we have that
$$
\widehat{W\ast\nu_\e}=2\pi\widehat W\widehat{\nu_\e}\in L^1(\R^2).
$$
Using these properties one can show that Plancherel Theorem holds for $W\ast\nu_\e$ and $\nu_\e$, that is,
\begin{equation}\label{pars-e}
\int_{\R^2}(W\ast\nu_\e)(x)\nu_\e(x)\, dx =2\pi\int_{\R^2}\widehat W(\xi)|\widehat{\nu_\e}(\xi)|^2\, d\xi
= 2\pi\int_{\R^2}\frac{\Om(\xi)}{|\xi|^2}|\widehat{\nu_\e}(\xi)|^2\, d\xi
\end{equation}
for every $\e>0$.

We now want to pass to the limit in \eqref{pars-e}, as $\e\to0^+$. For every $\xi\in\R^2$ we have
$$
\widehat{\nu_\e}(\xi)=2\pi\widehat\nu(\xi)\widehat\varphi(\e\xi) \ \to \ 2\pi\widehat\nu(\xi)\widehat\varphi(0)=\widehat\nu(\xi),
$$
as $\e\to0^+$. Therefore, $\Om(\xi)|\widehat{\nu_\e}(\xi)|^2/|\xi|^2$ converges to $\Om(\xi)|\widehat\nu(\xi)|^2/|\xi|^2$ for a.e.\ $\xi\in\R^2$, as $\e\to0^+$.
Moreover, $|\widehat\varphi(\e\cdot)|\leq C\|\varphi\|_{L^1}$. Either by dominated convergence or by Fatou's lemma we can thus pass to the limit
in the last integral in \eqref{pars-e}.

To pass to the limit in the left-hand side of \eqref{pars-e}, we observe that
$$
\int_{\R^2}(W\ast\nu_\e)(x)\nu_\e(x)\, dx =\int_{\R^2}(W\ast\varphi_\e\ast\varphi_\e)\ast\nu\, d\nu,
$$
where we used that $\varphi_\e$ is a radial, hence even, function. Set $\psi_\e:=\varphi_\e\ast\varphi_\e$ and note that $\psi_\e$ has the same properties as $\varphi_\e$ (it is radial, compactly supported, non-negative, and with unit integral). Since $W$ is continuous as a function with values in $\R\cup\{+\infty\}$, we have that
$(W\ast\psi_\e)(x)\to W(x)$, as $\e\to0^+$, for every $x\in\R^2$. Moreover, by \eqref{bddW}
\begin{equation}\label{bddW2}
(W\ast\psi_\e)(x) \leq (W_0\ast\psi_\e)(x)+C_2 \quad \text{ for every } x\in\R^2.
\end{equation}
Since $W_0$ is superharmonic and $\psi_\e$ is radial, we have
\begin{equation}\label{bddW3}
(W_0\ast\psi_\e)(x) \leq W_0(x) \leq W(x)- C_1 \quad \text{ for every } x\in\R^2,
\end{equation}
where the last inequality follows again from \eqref{bddW}.

Let $M>0$ be such that $\supp\nu\subset B_M(0)$ and let $c_M$ be the minimum of $W$ on $B_{4M}(0)$.
Combining \eqref{bddW2} and \eqref{bddW3}, we deduce that
\begin{equation}\label{bddW4}
0\leq (W\ast\psi_\e)(x) -c_M\leq W(x)+C_2- C_1 -c_M \quad \text{ for every } x\in B_{2M}(0)
\end{equation}
and for every $\e>0$ small enough.

We now write
\begin{eqnarray*}
\int_{\R^2}(W\ast\psi_\e)\ast\nu\, d\nu & = &
\iint_{\R^2\times\R^2}(W\ast\psi_\e)(x-y)\,d\mu_0(x)d\mu_0(y) 
\\
& & + \iint_{\R^2\times\R^2}(W\ast\psi_\e)(x-y)\,d\mu_1(x)d\mu_1(y) 
\\
& &
-2 \iint_{\R^2\times\R^2}(W\ast\psi_\e)(x-y)\, d\mu_0(x)d\mu_1(y).
\end{eqnarray*}
Since $\mu_0$ and $\mu_1$ have finite interaction energy, we can pass to the limit in the first two integrals on the right-hand side
by \eqref{bddW4} and dominated convergence. As for the last integral, it goes to the limit either by dominated convergence or by Fatou's lemma.
This completes the proof of \eqref{pars}.

To conclude, assume that 
\begin{equation}\label{pars-eq}
0=\int_{\R^2}W\ast\nu\, d\nu =2\pi\int_{\R^2}\frac{\Om(\xi)}{|\xi|^2}|\widehat\nu(\xi)|^2\, d\xi.
\end{equation}
Since $\Om$ cannot be identically zero (otherwise, $W$ would be constant) and is continuous outside~$0$, there exist $\delta>0$,
$\eta>0$, and $\xi_0\neq0$ such that $\Om(\xi)>\eta$ for all $\xi\in B_\delta(\xi_0)$. Thus, \eqref{pars-eq} implies that $\widehat\nu=0$ on $B_\delta(\xi_0)$. 
On the other hand, $\nu$ is a distribution with compact support, so by Paley-Wiener Theorem $\widehat\nu$ is the restriction to $\R^2$ of an entire function.
Therefore, $\widehat\nu$ has to be identically zero on $\R^2$, that is, $\nu=0$.
\end{proof}

\end{section}

\begin{section}{characterisation of minimisers}\label{sec:ch}
Throughout this section we consider the confinement $V(x)=|x|^2$ and we focus on the characterisation of the minimiser of $I$ for this specific choice of the confinement.

\subsection{The Coulomb case}
In the purely Coulomb case with confinement $V(x)=|x|^2$ the energy
\begin{equation}\label{I0}
I_0(\mu)=-\iint_{\R^2\times\R^2}\log|x-y|\,d\mu(y)d\mu(x)+\int_{\R^2}|x|^2\,d\mu(x), \quad \mu\in\P(\R^2),
\end{equation}
is invariant by rotations. Therefore, by uniqueness the minimiser $\mu_0$ has to be invariant by rotations, too. 
If we formally take the Laplacian of both sides of \eqref{EL1} (with $W_0$ in place of $W$), we obtain
$$
0=\Delta W_0\ast\mu_0+2=-2\pi\mu_0+2 \quad \text{ on } \supp\mu_0,
$$
where the last equality follows from the relation $\Delta W_0=-2\pi\delta_0$. Hence, $\mu_0=1/\pi$ on its support.
Since $\mu_0$ is a probability measure, one can conjecture that $\mu_0$ has to be the normalised characteristic function of $B_1(0)$.
This is indeed the case, as shown in the next theorem.

\begin{theorem}[\cite{G, M, Gk}]\label{GM}
The unique minimiser of the energy $I_0$ in \eqref{I0} is the so-called {\em circle law}, that is, the measure
$$
\mu_0=\frac1\pi\chi_{B_1(0)}.
$$
\end{theorem}

The proof of this result is based on the Gauss averaging principle: for any $r>0$
\begin{equation}\label{Gauss}
-\frac1{2\pi}\int_{-\pi}^\pi\log |z-re^{i\theta}|\, d\theta =
\begin{cases}
-\log r & \text{ if } |z|<r,
\\
-\log |z| & \text{ if } |z|\geq r,
\end{cases}
\end{equation}
where $z$ is a complex variable. A notable consequence of this principle is the well-known fact that 
the Coulomb potential
due to a homogeneous spherical body is the same, outside the body, as if all the mass were at its center.

Proving \eqref{Gauss} is straightforward: if $|z|>r$, the map $w\mapsto -\log|z-w|$ is harmonic for $|w|\leq r$, so \eqref{Gauss} follows by the mean value property. If $|z|<r$, one can write
$$
-\log |z-re^{i\theta}| = - \log |ze^{-i\theta} -r| = - \log |\bar ze^{i\theta} -r|. 
$$
The map $w\mapsto -\log|w -r|$ is harmonic for $|w|\leq |z|$, so one can conclude again by applying the mean value property. 
Finally, for $|z|=r$ one can argue by approximation with radii $\rho\to r^-$ and dominated convergence.

\begin{proof}[Proof of Theorem~\ref{GM}]
By applying the Gauss averaging principle we can explicitely compute the Coulomb potential of $\mu_0$: indeed, using polar coordinates we have
$$
(W_0\ast\mu_0)(x) =  -\frac1\pi \int_{B_1(0)}\log|x-y|\, dy =
-\frac1\pi \int_0^1\int_{-\pi}^\pi \log|x-r e^{i\theta}|\, d\theta\, r\,dr.
$$
By \eqref{Gauss} we deduce that
$$
(W_0\ast\mu_0)(x) =  
\begin{cases}
\frac12-\frac12|x|^2 & \text{ for } |x|\leq 1,
\\
-\log|x| & \text{ for } |x|>1.
\end{cases}
$$
Using this formula, one can immediately check that $(W_0\ast\mu_0)(x)+\frac12 |x|^2=\frac12$ for $x\in B_1(0)$ and
$(W_0\ast\mu_0)(x)+\frac12 |x|^2\geq \frac12$ for $x\not\in B_1(0)$, that is, $\mu_0$ satisfies the Euler-Lagrange equations.
By the results of the previous section this is enough to conclude that $\mu_0$ is the unique minimiser of~$I_0$.
\end{proof}

\subsection{The anisotropic case}
In this section we discuss the characterisation of the minimiser for the energy
\begin{equation}\label{Ikappa}
I(\mu)=\iint_{\R^2\times\R^2}\big(-\log|x-y|+\kappa(x-y)\big)\,d\mu(y)d\mu(x)+\int_{\R^2}|x|^2\,d\mu(x),
\end{equation}
where $\kappa$ is even, $0$-homogeneous, and of regularity $H^s(\mathbb S^1)$ with $s>3/2$. We start by considering the case where the Fourier transform of the interaction kernel (computed in \eqref{def:Omega}) is strictly positive outside~$0$.

\begin{theorem}\label{thm:CS}
Assume $\Om>0$ on $\mathbb S^1$, where $\Om$ is the function introduced in \eqref{def:Omega}.
Then the unique minimiser of the energy $I$ in \eqref{Ikappa} is given by the normalised characteristic function of the domain enclosed by an ellipse centered at the origin, whose
semi-axes $a_1,a_2$ satisfy the relation $a_1^2+a_2^2=2$.
\end{theorem}

Theorem~\ref{thm:CS} was originally proved by Carrillo\and Shu in \cite{CS22}. Here we present an alternative proof inspired by \cite{MMRSV22p}.

\begin{proof}[Proof of Theorem~\ref{thm:CS}]
We write a general domain enclosed by an ellipse as 
%
$E=RE_0$, where $R\in SO(2)$ and
$$
E_0=\Big\{ x\in\R^2: \ \frac{x_1^2}{a_1^2}+\frac{x_2^2}{a_2^2}\leq 1\Big\}
$$
and we set $\chi:=\chi_E/|E|$. The theorem is proved if we show that there exist $a_1,a_2>0$ and $R\in SO(2)$
such that
\begin{equation}\label{EL1-p}
(W\ast\chi)(x)+\frac12 |x|^2=c \quad \text{for every } x\in E,
\end{equation}
and
\begin{equation}\label{EL2-p}
(W\ast\chi)(x)+\frac12 |x|^2\geq c \quad \text{for every } x\in \R^2\setminus E
\end{equation}
for some constant $c\in\R$.

Note that $W\ast\chi$ is a $C^1$ function in $\R^2$. To compute its expression we would like to make use of the inversion formula, that is, write $W\ast\chi$ as an integral of its Fourier transform. However, $\widehat W$ is not a function, it is only a tempered distribution. To circumvent this difficulty we apply this strategy to the gradient of $W\ast\chi$.
In fact, the Fourier transform of $\nabla(W\ast\chi)$ is given by
$$
i\xi\widehat{W\ast\chi}(\xi)=2\pi i\xi\widehat W(\xi)\widehat\chi(\xi),
$$
so the presence of the factor $\xi$ annihilates the singular part of $\widehat W$.
Note that $\nabla(W\ast\chi)$ is a continuous function that behaves as $1/|x|$ at infinity, so it is a tempered distribution.\smallskip

\noindent
{\it Step 1: The inversion formula for $\nabla(W\ast\chi)$.} We recall that
$$
\widehat{\chi_{B_1(0)}}(\xi)=\frac{J_1(|\xi|)}{|\xi|},
$$
where $J_1$ denotes the Bessel function of the first kind of order $1$. There is no explicit formula for $J_1$; however, it is well known that
$J_1$ has the following behaviour:
\begin{equation}\label{J1-1}
J_1(|\xi|)\simeq \frac12|\xi| \quad \text{ as } \xi\to0,
\end{equation}
\begin{equation}\label{J1-2}
|J_1(|\xi|)|\leq \frac{C}{|\xi|^{1/2}} \quad \text{ as } |\xi|\to\infty,
\end{equation}
see, e.g., \cite[Section~5.16]{Lebedev}.

For $a=(a_1,a_2)\in\R^2$ we denote the diagonal matrix ${\rm diag}(a_1,a_2)$ by $D(a)$.
By writing $x\in E$ as $x=RD(a)y$ with $y\in B_1(0)$, we obtain
$$
\widehat\chi(\xi)=\frac1\pi\widehat{\chi_{B_1(0)}}\big(D(a)R^T\xi\big)=\frac1\pi \frac{J_1(|D(a)R^T\xi|)}{|D(a)R^T\xi|}.
$$
By the bounds \eqref{J1-1}--\eqref{J1-2} we deduce that the function $2\pi i\xi\widehat W\widehat\chi$ belongs to $L^1(\R^2)$. Therefore, since 
$\nabla(W\ast\chi)$ is a continuous function, the inversion formula holds, that is,
\begin{equation*}
\nabla(W\ast\chi)(x)=\int_{\R^2} i\xi\frac{\Om(\xi)}{|\xi|^2}\widehat\chi(\xi)e^{ix\cdot\xi}\, d\xi
= -\int_{\R^2}\xi\frac{\Om(\xi)}{|\xi|^2}\widehat\chi(\xi)\sin(x\cdot\xi)\, d\xi
\end{equation*}
for every $x\in\R^2$, where we used that $\Om$ and $\widehat\chi$ are even functions.
Passing to polar coordinates we obtain
\begin{eqnarray*}
\nabla(W\ast\chi)(x) & = & -\int_{\mathbb S^1}\int_0^\infty y\Om(y)\widehat\chi(\rho y)\sin(\rho x\cdot y)\, d\rho\, d\HH^1(y)
\\
& = & -\frac1\pi \int_{\mathbb S^1}\int_0^\infty y \Om(y) \frac{J_1(\rho |D(a)R^Ty|)}{\rho |D(a)R^Ty|} \sin(\rho x\cdot y)\, d\rho \, d\HH^1(y).
\end{eqnarray*}
Setting $r:=\rho |D(a)R^Ty|$, we deduce that
\begin{equation}\label{inv-form-c}
\nabla(W\ast\chi)(x)=-\frac1\pi \int_{\mathbb S^1}\frac{y \Om(y)}{|D(a)R^Ty|} \int_0^\infty  \frac{J_1(r)}{r} \sin(r \alpha(x,y))\, dr\, d\HH^1(y)
\end{equation}
for every $x\in\R^2$, where we set
\begin{equation}\label{def:alfa}
\alpha(x,y):=\frac{x\cdot y}{|D(a)R^Ty|}.
\end{equation}
Using formula (5), page 99, in \cite{Bateman}, one can compute the improper integral
\begin{equation}\label{improp}
\int_0^\infty  \frac{J_1(r)}{r} \sin(r \alpha)\, dr=\begin{cases}
\alpha & \text{ if } 0\leq\alpha\leq 1,
\\
\dfrac1{\alpha+\sqrt{\alpha^2-1}} & \text{ if } \alpha>1.
\end{cases}
\end{equation}
Moreover, we have that, if $x\in E$, then
$$
|x\cdot y|=|D(a)^{-1}R^Tx\cdot D(a)R^Ty|\leq |D(a)R^Ty|,
$$
that is, $|\alpha(x,y)|\leq 1$ for every $x\in E$. We conclude that
\begin{equation}\label{inv-form}
\nabla(W\ast\chi)(x)=-\frac1\pi \int_{\mathbb S^1} \frac{\Om(y)}{|D(a)R^Ty|} \alpha(x,y)y \, d\HH^1(y)
= -\frac1\pi \int_{\mathbb S^1} \frac{\Om(y)}{|D(a)R^Ty|^2}(x\cdot y)y \, d\HH^1(y)
\end{equation}
for every $x\in E$. Formula \eqref{inv-form} shows that $\nabla(W\ast\chi)$ is a homogeneous polynomial of degree~$1$ inside $E$. In other words, up to an additive constant, 
$W\ast\chi$ is a homogeneous polynomial of degree~$2$ inside~$E$.\smallskip

\noindent
{\it Step 2: Solving the first Euler-Lagrange equation.} Solving \eqref{EL1-p} corresponds to finding $a_1,a_2>0$ and $R\in SO(2)$
such that
\begin{equation*}
\nabla(W\ast\chi)(x) +x = 0 \quad \text{for every } x\in E.
\end{equation*}
By formula \eqref{inv-form} this translates into the following system of three equations:
\begin{equation}\label{system}
\frac1\pi \int_{\mathbb S^1} \frac{\Om(y)}{|D(a)R^Ty|^2} y_j y_k\, d\HH^1(y)=\delta_{jk} \quad \text{ for every } j,k=1,2,
\end{equation}
where $\delta_{jk}$ is the Kronecker delta. 

Let us denote by $\mpos$ the set of positive definite $2\times2$ symmetric matrices. Note that $M:=RD(a)^2R^T\in\mpos$ and
$|D(a)R^Ty|^2=My\cdot y$.
Solving the system \eqref{system} is therefore equivalent to finding $M\in \mpos$ such that
\begin{equation}\label{system1}
\frac1\pi \int_{\mathbb S^1} \frac{\Om(y)}{My\cdot y} y_j y_k \, d\HH^1(y)=\delta_{jk} \quad \text{ for every } j,k=1,2.
\end{equation}
Let us denote the entries of $M$ by $M_{jk}$ for $j,k=1,2$. 
Note that, if we multiply by $M_{jk}$ the $jk$ equation in \eqref{system} and we sum over $j,k$, we obtain
$$
\tr(M)=M_{11}+M_{22}=\frac1\pi  \int_{\mathbb S^1} \Om(y) \, d\HH^1(y)=2,
$$
where we used \eqref{int-Omega}. Since $\tr(M)=a_1^2+a_2^2$, where $a_1,a_2$ are the semi-axes of the candidate ellipse,
we deduce that, if a minimising ellipse exists, then necessarily its semi-axes satisfy the relation
$$
a_1^2+a_2^2=2.
$$

To solve \eqref{system1} it is convenient to introduce the function 
$$
f(M):=-\frac1\pi \int_{\mathbb S^1}\Om(y)\log(My\cdot y) \, d\HH^1(y) +\tr(M)
$$
defined for every $M\in\mpos$.
It is immediate to see that conditions \eqref{system1} are satisfied if we find $M_0\in\mpos$ such that $\nabla_{\!M} f(M_0)=0$,
where with a slight abuse of notation we denoted by $\nabla_{\!M}$ the operator
$$
\nabla_{\!M}=\Big(\frac{\partial}{\partial M_{11}}, \frac{\partial}{\partial M_{22}}, \frac{\partial}{\partial M_{12}}\Big).
$$

We claim that $f$ has a minimiser in $\mpos$. Since $\mpos$ is an open set, the claim will imply the existence of a critical point of $f$ in this set and thus,
of a solution to \eqref{system1}.

Given $M\in\mpos$, we first look at the behaviour of the function $f$ along the line $tM$ for $t>0$, that is, we consider the function
$$
g(t):=f(tM)=-\frac1\pi \int_{\mathbb S^1}\Om(y)\log(My\cdot y) \, d\HH^1(y) -2\log t+t \tr(M)
$$
for $t>0$, where the last equality follows from \eqref{int-Omega}. It is immediate to see that $g$ is minimised at $t_\ast=2/\tr(M)$.
Therefore, minimising $f$ over $\mpos$  is equivalent to minimising $f$ on the set
$$
\mathcal M:=\big\{M\in\mpos: \ \tr(M)=2\big\}.
$$
By diagonalization any $M\in\mathcal M$ can be written as $M=QD(b)Q^T$ with $Q\in SO(2)$ and $b=(\beta,2-\beta)$, $\beta\in(0,2)$.
Using this representation and a change of variables we have that
$$
f(M)=f(QD(b)Q^T)=- \frac1\pi \int_{\mathbb S^1}\Om(Qy)\log(\beta y_1^2+(2-\beta)y_2^2) \, d\HH^1(y) +2.
$$
Therefore, setting
$$
\gamma(\beta, Q):=- \frac1\pi \int_{\mathbb S^1}\Om(Qy)\log(\beta y_1^2+(2-\beta)y_2^2) \, d\HH^1(y)
$$
for every $\beta\in[0,2]$ and $Q\in SO(2)$, it is enough to show that $\gamma$ has a minimiser in $(0,2)\times SO(2)$ to conclude that $f$ has a minimiser in $\mathcal M$ and thus in $\mpos$.

The function $\gamma$ is finite and continuous on the compact set $[0,2]\times SO(2)$. Therefore, it has a minimiser $(\beta_0,Q_0)$ in this set.
It remains to prove that $\beta_0$ is neither $0$ nor $2$. 

By assumption there esists a constant $C_0>0$ such that
$$
\Om(\xi)\geq C_0 \quad \text{ for every } \xi\in\mathbb S^1.
$$
Using this inequality and \eqref{int-Omega}, we obtain that
\begin{eqnarray*}
\frac{\partial}{\partial\beta}\gamma(\beta, Q_0) & = &
\frac1\pi \int_{\mathbb S^1}\Om(Q_0y)\frac{2y_2^2-1}{\beta y_1^2+(2-\beta)y_2^2} \, d\HH^1(y)
\\
& \leq & \frac4{2-\beta} - \frac1\pi \int_{\mathbb S^1}\frac{C_0}{\beta y_1^2+(2-\beta)y_2^2} \, d\HH^1(y)
\end{eqnarray*}
for every $\beta\in(0,2)$.
Since the right-hand side goes to $-\infty$ as $\beta\to0^+$, we deduce that there exists some $\delta>0$
such that
$$
\frac{\partial}{\partial\beta}\gamma(\beta, Q_0) <0
$$
for every $\beta\in(0,\delta)$. Hence, $\beta_0$ cannot be equal to zero.
One can show in a similar way that $\beta_0$ cannot be $2$ either.
This concludes the proof of the step.\smallskip

{\it Step 3: The first Euler-Lagrange equation implies the second one.}
To complete the proof, we show that, if $\chi=\chi_E/|E|$ satisfies \eqref{EL1-p}, then it satisfies \eqref{EL2-p}, as well.
Assume $\chi$ satisfies \eqref{EL1-p}. By \eqref{EL1-p} it is enough to prove that the potential $(W\ast\chi)(x)+\frac12 |x|^2$ grows in the radial direction, that is,
$$
\nabla(W\ast\chi)(x)\cdot x + |x|^2\geq0 \quad \text{for every } x\in \R^2\setminus E.
$$

Let $x\in \R^2\setminus E$. We can write $x=tx^0$ with $x^0\in E$. By \eqref{EL1-p} we have that
$$
\nabla(W\ast\chi)(x^0)\cdot x^0 + |x^0|^2 = 0,
$$
which can be written by \eqref{inv-form} as
$$
-\frac1\pi \int_{\mathbb S^1} \Om(y)\alpha^2(x^0,y)\, d\HH^1(y) + |x^0|^2 = 0,
$$
where $\alpha$ is defined in \eqref{def:alfa}.
Multiplying this equation by $t^2$ yields
\begin{equation}\label{zero}
-\frac1\pi \int_{\mathbb S^1} \Om(y)\alpha^2(x,y)\, d\HH^1(y) + |x|^2 = 0.
\end{equation}
By the inversion formula \eqref{inv-form-c} and \eqref{improp} we can write
\begin{eqnarray*}
\nabla(W\ast\chi)(x)\cdot x +|x|^2 & = & |x|^2-\frac1\pi \int_{\mathbb S^1}\Om(y)\alpha^2(x,y)\chi_{[-1,1]}(\alpha(x,y))\, d\HH^1(y) 
\\
& &
-\frac1\pi \int_{\mathbb S^1}\Om(y)\frac{|\alpha(x,y)|}{|\alpha(x,y)|+\sqrt{\alpha^2(x,y)-1}}\chi_{\R\setminus[-1,1]}(\alpha(x,y))\, d\HH^1(y).
\end{eqnarray*}
Owing to \eqref{zero}, the expression above reduces to
$$
\nabla(W\ast\chi)(x)\cdot x +|x|^2 
= \frac1\pi \int_{\mathbb S^1}\Om(y)|\alpha(x,y)|\sqrt{\alpha^2(x,y)-1}\,\chi_{\R\setminus[-1,1]}(\alpha(x,y))\, d\HH^1(y) 
\geq 0.
$$
This concludes the proof.
\end{proof}

The next result shows how to combine the previous theorem and an approximation argument to characterise the minimiser
in the `degenerate' case where the Fourier transform of the interaction kernel is only non-negative outside~$0$.

\begin{cor}\label{cor:deg}
Assume $\Om\geq0$ on $\mathbb S^1$, where $\Om$ is the function introduced in \eqref{def:Omega}.
Then the unique minimiser of the energy $I$ in \eqref{Ikappa} is either the normalised characteristic function of the domain enclosed by an ellipse centered at the origin and with semiaxes $a_1,a_2$ satisfying $a_1^2+a_2^2=2$, or is the semicircle law \eqref{semicircle}, up to rotations.
\end{cor}

\begin{proof}
For $\e>0$ we consider the kernels
$$
W_\e(x)=-(1+\e)\log|x|+\kappa(x)
$$
and we denote by $I_\e$ the corresponding energies with confinement $|x|^2$. Arguing as in \eqref{hatk}, it is immediate to see that
$$
\widehat{W_\e}(\xi)=c_\e\delta_0+\frac{\Om(\xi)+\e}{|\xi|^2}
$$
with $\Om+\e>0$ on $\mathbb S^1$. By Theorem~\ref{thm:CS} for every $\e>0$ the unique minimiser $\mu_\e$ of $I_\e$ is the normalised characteristic function of the domain enclosed by an ellipse
of semiaxes $a_{1,\e}$ and $a_{2,\e}$ with $(a_{1,\e})^2+(a_{2,\e})^2=2$. The sequence $(\mu_\e)_\e$ is tight, since the support of $\mu_\e$ is contained in the closed ball of center $0$ and radius $\sqrt2$ for every $\e>0$. Therefore, up to subsequences, $(\mu_\e)_{\e}$ converges narrowly to a measure $\mu_0\in\P(\R^2)$.
We claim that $\mu_0$ is the minimiser of $I$. Indeed, since the supports of $\mu_\e$ and $\mu_0$ are uniformly bounded, on these sets $W_0$ is bounded from below by some constant $-c_0$, with $c_0>0$. Therefore,
$$
I_\e(\mu_\e)\geq I(\mu_\e)- c_0\e
$$
and by lower semicontinuity
$$
\liminf_{\e\to0^+} I_\e(\mu_\e)\geq \liminf_{\e\to0^+} I(\mu_\e) \geq I(\mu_0).
$$
On the other hand, if $\mu$ is any measure in $\P(\R^2)$ with compact support and such that 
\begin{equation}\label{limsup-asp}
\iint_{\R^2\times\R^2}\big(-\log|x-y|+\kappa(x-y)\big)\,d\mu(y)d\mu(x)<+\infty,
\end{equation}
then by minimality
$$
\limsup_{\e\to0^+} I_\e(\mu_\e)\leq \lim_{\e\to0^+} I_\e(\mu) = I(\mu).
$$
Therefore, $\mu_0$ minimises $I$ over all measures with compact support and satisfying \eqref{limsup-asp}. 
By Theorem~\ref{thm:exist} and Proposition~\ref{prop:pars} we conclude that $\mu_0$ is the minimiser of $I$ on the whole $\mathcal P(\R^2)$.

Up to subsequences, we can assume that the semi-axes converge, that is, $a_{1,\e} \to a_1$ and $a_{2,\e}\to a_2$, as $\e\to0^+$, with $a_1,a_2\geq0$ and $a_1^2+a_2^2=2$.
Therefore, we have two cases: either both $a_1$ and $a_2$ are strictly positive, or one of them is $0$ and the other one is $\sqrt2$.
In the first case $\mu_0$ is the normalised characteristic function of the domain enclosed by an ellipse with semiaxes $a_1,a_2$ satisfying $a_1^2+a_2^2=2$.
In the second case $\mu_0$ is the semicircle law \eqref{semicircle} on some line passing through the origin.
\end{proof}

\begin{remark}
The proofs of Theorem~\ref{thm:CS} and Corollary~\ref{cor:deg} extend to $\R^3$, see \cite{MMRSV22p}, where the kernel $W$ is given by the following anisotropic variant of the three-dimensional Coulomb kernel:
$$
W(x)=\frac{\Psi(x)}{|x|} \quad \text{ for } x\in\R^3, x\neq0, \qquad W(0)=+\infty,
$$
with $\Psi$ strictly positive, even, $0$-homogeneous, and smooth enough on $\mathbb S^2$. More precisely, we proved that, if $\widehat W>0$, then the unique minimiser of the energy
$$
I_{3\rm d}(\mu)=\iint_{\R^3\times\R^3}W(x-y)\,d\mu(y)d\mu(x)+\int_{\R^3}|x|^2\,d\mu(x), \quad \mu\in\P(\R^3),
$$ 
is the normalised characteristic function of the domain enclosed by an ellipsoid centered at the origin. If $\widehat W\geq0$, the minimiser is either of this form or it may collapse on a two-dimensional measure, whose support is given by the domain enclosed by an ellipse.
\end{remark}

\begin{remark}
The original proof by Carrillo \& Shu in \cite{CS22} is based on a different technique. Their starting point is a formula expressing the interaction kernel as an integral of one-dimensional logarithmic kernels on projections. More precisely, they show that 
$$
-\log|x|+\kappa(x) = -\int_{-\pi}^\pi\Om(e^{i\varphi})\log|x\cdot e^{i\varphi}|\, d\varphi + \text{ constant,}
$$
where $\Om$ is the same function as in \eqref{def:Omega}.
They observe that the projection of $\chi_E/|E|$ (with $E$ the domain enclosed by an ellipse centered at the origin) on any line passing through the origin is a semicircle law.
Using this remark and the representation formula above, they argue by projection and conclude by
applying the minimality of the semicircle law for the one-dimensional logarithmic kernel. The same approach can be used to treat two-dimensional anisotropic Riesz kernels of the form
$$
\frac{\Psi(x)}{|x|^s}, \quad 0<s<2,
$$
with $\Psi$ strictly positive, even, $0$-homogeneous, and smooth enough on $\mathbb S^1$. However, their strategy of proof extends to $\R^3$ only under some rather strong symmetry assumptions on the anisotropy, that essentially reduce the problem to a two-dimensional setting (see \cite{CS3d}).
\end{remark}

\end{section}

\begin{section}{Related results and open questions}

In this last section we show how to explicitely determine the optimal distributions in some concrete cases and we discuss some open problems.

\subsection{The dislocation case}
As a first result, we show how to deduce the characterisation of the minimiser of $I_\alpha$ in \eqref{int:Ialfa} from the results of the previous sections.
The original proofs in \cite{CMMRSV, MRS} are based on completely different arguments.

\begin{theorem}\label{thm:disl}
Let $I_\alpha$ be the functional defined in \eqref{int:Ialfa}. Then for $|\alpha|<1$ the unique minimiser of $I_\alpha$ is given by the measure \eqref{E-alfa}.
For $\alpha\geq1$ the unique minimiser of $I_\alpha$ is the semicircle law \eqref{semicircle}, whereas for $\alpha\leq-1$ it is the semicircle law on the horizontal axis
\begin{equation}\label{hsc}
\frac1\pi\sqrt{2-x_1^2}\HH^1\mres[-\sqrt2,\sqrt2](x_1)\otimes \delta_0(x_2).
\end{equation}
\end{theorem}

\begin{proof}
We set $W_\alpha(x)=-\log|x|+\kappa_\alpha(x)$, where
$$
\kappa_\alpha(x):=\alpha \frac{x_1^2}{|x|^2}.
$$
Since we can write
$$
\kappa_\alpha (x)=\frac\alpha2 \frac{x_1^2-x_2^2}{|x|^2}+\frac\alpha2,
$$
Lemma~\ref{lemmaStein} ensures that 
\begin{equation}\label{Fourier Wa}
\widehat{W_\alpha}(\xi)= c_\alpha\delta_0+\frac{\Om_\alpha(\xi)}{|\xi|^2}
\end{equation}
with 
$$
\Om_\alpha(\xi)=(1-\alpha)\frac{\xi_1^2}{|\xi|^2}+(1+\alpha)\frac{\xi_2^2}{|\xi|^2}, \quad \xi\neq0.
$$ 

Assume $|\alpha|<1$. Since $\Om_\alpha>0$ on $\mathbb S^1$ in this case,
by Theorem~\ref{thm:CS} the unique minimiser is the normalised characteristic function of the domain enclosed by an ellipse of semiaxes $a_{1,\alpha}, a_{2,\alpha}$.
We note that $W_\alpha$ is symmetric with respect to the coordinate axes, that is,
$$
W_\alpha(-x_1,x_2)=W_\alpha(x_1,-x_2)=W_\alpha(x) \quad \text{ for every } x\in\R^2.
$$
By uniqueness the minimiser must have the same symmetry, that is, the ellipse is symmetric with respect to the coordinate axes. Therefore, in system \eqref{system} 
the rotation $R$ is necessarily the identity matrix and the equation for $j=1$, $k=2$ is trivially satisfied by symmetry. 
In other words, using the expression of $\Om_\alpha$, the semiaxes $a_{1,\alpha}, a_{2,\alpha}$ satisfy
$$
\frac1\pi \int_{\mathbb S^1} \frac{(1-\alpha)y_1^2+(1+\alpha)y_2^2}{(a_{1,\alpha})^2 y_1^2 + (a_{2,\alpha})^2 y_2^2} y_j^2\, d\HH^1(y)=1 \quad \text{ for every } j=1,2.
$$
It is immediate to see that $a_{1,\alpha}=\sqrt{1-\alpha}$, $a_{2,\alpha}=\sqrt{1+\alpha}$ is a solution of this system. It is indeed the unique solution, since we proved in Step~3 of the proof of Theorem~\ref{thm:CS} that any solution of \eqref{system} is automatically a minimiser and the minimiser is unique. We conclude that for $|\alpha|<1$ the minimiser is given by the measure \eqref{E-alfa}.

Let now $\alpha=1$. Since $\Om_1\geq0$ on $\mathbb S^1$, the minimiser is still unique. Arguing as in the proof of Corollary~\ref{cor:deg}, one can show that the minimiser of $I_1$ has to be the limit of the minimiser of $I_\alpha$, as $\alpha\to1^-$. Therefore, the minimiser is the semicircle law \eqref{semicircle}. The same argument applies to $\alpha=-1$.

If $\alpha>1$, let us denote by $\mu_{\rm sc}$ the semicircle law in \eqref{semicircle} and let $\mu$ be any other measure in $\P(\R^2)$. Then
$$
I_{\alpha}(\mu)\geq I_{1}(\mu)>I_1(\mu_{\rm sc})=I_\alpha(\mu_{\rm sc}),
$$
where the last equality follows from the fact that the anisotropic interaction is $0$ on the semicircle law. A similar argument shows that \eqref{hsc} is the unique minimiser of $I_\alpha$ for $\alpha<-1$.
\end{proof}

\begin{remark}\label{ross}
For the anisotropy \eqref{kappa-aniso} and quadratic confinement one can prove that the unique minimiser is the normalised characteristic function of the domain enclosed by an ellipse (with explicit semi-axes) if $b^2<1+a^2$, whereas is the semicircle law \eqref{semicircle} if $b^2\geq 1 +a^2$, see \cite{Giorgio}.
Since the Fourier transform of the kernel $W$ in this case is `degenerate' outside $0$ for any $b>a>0$, this example shows that both options predicted by Corollary~\ref{cor:deg} can indeed occur. From a mechanical viewpoint the quantity $H=b^2-a^2-1$ is the so-called anisotropy factor, which measures the degree of anisotropy in a cubic crystal, see \cite[eq.~(13-27)]{HL}.
\end{remark}

\subsection{Elliptic physical confinement}
In this section we provide a full characterisation of minimisers for a general anisotropy in the case of the physical confinement 
$$
V(x)=\begin{cases}
0 & \text{ if } x\in E, \\
+\infty & \text{ if } x\not\in E,
\end{cases}
$$
where $E$ is the domain enclosed by an ellipse centered at the origin. More precisely, we have the following result.

\begin{theorem}\label{thm:phys}
Let $E=RE_0$ with $R\in SO(2)$ and 
$$
E_0=\Big\{ x\in\R^2: \ \frac{x_1^2}{a_1^2}+\frac{x_2^2}{a_2^2}\leq 1\Big\}.
$$
Let $J$ be the functional defined by
$$
J(\mu)=\iint_{E\times E}\big(-\log|x-y|+\kappa(x-y)\big)\,d\mu(y)d\mu(x)
$$
for every $\mu\in \P(E)$, where the anisotropy $\kappa$ is even, $0$-homogeneous, and of class $H^s$ on $\mathbb S^1$ with $s>3/2$.
Assume $\Om\geq0$ on $\mathbb S^1$, where $\Om$ is the function introduced in \eqref{def:Omega}.
Then the unique minimiser of $J$ is given by the push-forward $f_\#\mu_{E_0}$ of the measure
$$
\mu_{\partial E_0}:=\frac1{2\pi a_1a_2}\Big(\frac{x_1^2}{a_1^4}+\frac{x_2^2}{a_2^4}\Big)^{-\frac12}\HH^1\mres \partial E_0
$$
through the rotation map $f(x)=Rx$ for $x\in\R^2$.
\end{theorem}

\begin{remark}
This result was proved in \cite{MS} for a special class of anisotropies. Note that when $E$ is the closed ball $B_r$ of radius $r>0$ and center $0$ the minimiser is given by the
uniform distribution on the boundary of the ball, that is, by the measure
$$
\mu_{\partial B_r}:=\frac1{2\pi r}\HH^1\mres \partial B_r.
$$
\end{remark}

\begin{proof}[Proof of Theorem~\ref{thm:phys}]
Up to rotating the axes and replacing $\kappa(x)$ by $\kappa(Rx)$, we can assume without loss of generality that $R$ is the identity matrix and $E=E_0$.

Theorem~\ref{thm:exist} and Proposition~\ref{prop:pars} guarantee that the minimiser exists and is unique. Moreover, it is characterised by the Euler-Lagrange equations \eqref{EL1}--\eqref{EL2}, which take the following form:
\begin{eqnarray}\label{EL1p}
& & (W\ast\mu)(x)=c \quad \text{ for } \mu\text{-a.e.\ } x\in\supp\mu,
\\
\label{EL2p}
& & (W\ast\mu)(x)\geq c \quad \text{ for q.e.\ } x\in E_0.
\end{eqnarray}
To prove the theorem it is enough to show that the measure $\mu_{\partial E_0}$ satisfies \eqref{EL1p}--\eqref{EL2p}.

For $t>0$ we consider the set
$$
E_t=\Big\{ x\in\R^2: \ \frac{x_1^2}{a_1^2}+\frac{x_2^2}{a_2^2}\leq 1+t\Big\}.
$$
By the formula \eqref{inv-form} for the potential of a general ellipse we have that
\begin{equation}\label{invar}
\nabla(W\ast\chi_{E_t})(x)= - a_1a_2 \int_{\mathbb S^1} \frac{\Om(y)}{|D(a)y|^2}(x\cdot y)y \, d\HH^1(y)
\end{equation}
for every $x\in E_t$ and every $t\geq0$. 

Let now $x$ be in the interior of $E_0$. Since $x\in E_t$ for every $t\geq0$, equation \eqref{invar} holds for every $t\geq0$. Differentiating both sides of \eqref{invar} with respect to $t$ and applying the coarea formula yield
$$
0= \frac{d}{dt}\Big( \int_{E_t} \nabla W(x-y)\, dy\Big) 
= \frac12 \int_{\partial E_t} \nabla W(x-y)\Big(\frac{y_1^2}{a_1^4}+\frac{y_2^2}{a_2^4}\Big)^{-\frac12}\, d\HH^1(y)
$$
for a.e.\ $t\geq0$. For $x$ in the interior of $E_0$ the right-hand side is a continuous function of $t$, therefore we deduce that
$$
0= \frac12 \int_{\partial E_0} \nabla W(x-y)\Big(\frac{y_1^2}{a_1^4}+\frac{y_2^2}{a_2^4}\Big)^{-\frac12}\, d\HH^1(y)=\pi a_1a_2\nabla(W\ast\mu_{\partial E_0})(x)
$$
for every $x$ in the interior of $E_0$. Since $W\ast\mu_{\partial E_0}$ is a continuous function in $\mathbb R^2$, this implies that $W\ast\mu_{\partial E_0}$ is constant in $E_0$, that is, \eqref{EL1p}--\eqref{EL2p} are satisfied for $\mu=\mu_{\partial E_0}$.
\end{proof}

\subsection{Some open questions and further comments}
As a consequence of Theorem~\ref{thm:CS} and Corollary~\ref{cor:deg}, energies of the form \eqref{Ikappa} may have minimisers of non-full dimensionality only if the Fourier transform of their kernel 
is degenerate. However, degeneracy of the Fourier transform is not a sufficient condition for loss of dimensionality, see Remark~\ref{ross}.
The arguments of proof in Theorem~\ref{thm:CS} show that, in the convexity range, loss of dimensionality occurs if and only if system \eqref{system1} does not have a solution $M$ in $\mpos$.
Yet it would be desirable to devise a criterion for the occurrence of a lower dimensional optimal distribution, without resorting to explicit computations. 
Similarly, in the case of fully dimensional minimisers, there is no characterisation available for the rotation of the optimal ellipse with respect to the coordinate axes.

However, some simple considerations can be made in some specific cases. Indeed, if $\mu$ is a measure with no atoms and support on a straight line passing through the origin, then by $0$-homogeneity
$$
\iint_{\R^2\times\R^2} \kappa(x-y)\,d\mu(x) d\mu(y)=\kappa(v)
$$
where $v\in \mathbb S^1$ is a vector parallel to the support of $\mu$.  
Since the logarithmic interaction and the confinement are radially symmetric, this implies that, among all measures with support on a straight line through the origin, the minimal energy is attained in the directions where $\kappa$ is minimal. Therefore, if $\kappa$ has more than one minimiser in the set
$\{x\in\mathbb S^1: x_1\in (-1,1]\}$, then loss of dimensionality cannot occur, as long as the Fourier transform of the kernel is non-negative, otherwise uniqueness would be violated.


Theorem~\ref{thm:phys} shows that the choice of confinement may also have a strong impact on the shape of minimisers and on their dimensionality. Preliminary computations indicate that strict positivity of the Fourier transform should guarantee full dimensionality of minimisers for smooth confinements, such as, e.g., $V(x)=|x|^p$ with $p\geq2$.

Another interesting question is the analysis of optimal distributions outside the convexity range, that is, for kernels whose Fourier transform is negative along some directions.
Numerical simulations seem to suggest the occurrence of rather complex patterns, see \cite{CS22}.  
\end{section}

\bigskip\bigskip

\noindent
\textbf{Acknowledgements.}
Part of the material of this paper is based on the content of the course ``Nonlocal interaction problems in dislocation theory'', that the author taught at the Summer School on Analysis and Applied Mathematics in M\"unster in September 2022. The author would like to express her gratitude to the organisers and to the University of M\"unster for the support and the hospitality. The author thanks R\'emy Rodiac for bringing Bochner Theorem to her attention.
Support from MIUR--PRIN 2017 is also acknowledged. The author is a member of GNAMPA--INdAM.

\end{document}